\documentclass[10pt,twoside,leqno]{amsart}
\pdfoutput=1
\usepackage{amsfonts}
\usepackage{amsmath}
\usepackage{mathtools}
\usepackage{amsthm}
\usepackage{amssymb}
\usepackage{pifont}
\newcommand{\cmark}{\ding{51}}%
\usepackage{mathrsfs}
\usepackage[numbers]{natbib}
\usepackage[fit]{truncate}
\usepackage{geometry}
\usepackage{hyperref}
\usepackage{bm}
\usepackage{tikz-cd}
\usepackage[makeroom]{cancel}
\usepackage{xcolor,colortbl}
\usepackage{hhline}
\usepackage{float}
\usepackage{enumitem} 
\usepackage{stackengine}
\usepackage{array, multirow}
\usepackage{soul}

\newcommand\xrowht[2][0]{\addstackgap[.5\dimexpr#2\relax]{\vphantom{#1}}}
\newcommand{\R}{\mathbb{R}}
\newcommand{\C}{\mathbb{C}}
\newcommand{\upperset}[2]{\underset{\text{\raisebox{1ex}{\smash{\fontsize{5}{5}$#1$}}}}{#2}}

\newcommand{\delb}{\overline\partial}

\theoremstyle{plain}
\newtheorem{theorem}{Theorem}[section]

\newtheorem{lemma}[theorem]{Lemma}

\newtheorem{proposition}[theorem]{Proposition}

\theoremstyle{definition}
\newtheorem{definition}[theorem]{Definition}
\newtheorem{example}[theorem]{Example}

\newtheorem{remark}[theorem]{Remark}

\begin{document}

\title{Generalized K\"ahler almost abelian Lie groups}

\author{Anna Fino}
\address[A. Fino]{Dipartimento di Matematica\\
 Universit\`a di Torino\\
Via Carlo Alberto 10\\
10123 Torino, Italy} \email{annamaria.fino@unito.it}

\author{Fabio Paradiso}
\address[F. Paradiso]{Dipartimento di Matematica\\
 Universit\`a di Torino\\
Via Carlo Alberto 10\\
10123 Torino, Italy} \email{fabio.paradiso@unito.it}

\subjclass[2010]{53D18, 53C15, 53C30, 53C55}
\keywords{Almost abelian Lie groups, Hermitian metrics, Generalized K\"ahler structures, Holomorphic Poisson structures}

\begin{abstract}
We study  left-invariant generalized K\"ahler structures on almost abelian Lie groups, i.e.,  on  solvable Lie groups with a codimension-one abelian normal subgroup.   In particular, we  classify  six-dimensional   almost abelian Lie groups which admit a left-invariant complex structure    and establish which of those have a left-invariant Hermitian  structure whose  fundamental 2-form is $\partial \overline \partial$-closed.  We obtain a classification of six-dimensional generalized K\"ahler almost abelian Lie groups and    determine   the six-dimensional  compact  almost abelian solvmanifolds admitting an invariant  generalized K\"ahler structure. Moreover,  we  prove some results in relation to the existence of holomorphic Poisson structures and to the pluriclosed flow. \end{abstract}

\maketitle

\section{Introduction}

Generalized K\"ahler structures were introduced and studied by Gualtieri \cite{Gua,Gua1} in the more general context of generalized geometry started by Hitchin in  \cite{Hit1}.

Recall that a {\em generalized K\"ahler structure} on a $2n$-dimensional manifold $M$ is a pair of commuting complex structures $({\mathcal J}_1, {\mathcal J}_2)$ on the vector bundle $TM \oplus T^*M$, which are  integrable with respect to the (twisted) Courant bracket on $TM \oplus T^*M$,  are compatible with the natural inner-product  $\langle \cdot, \cdot \rangle$ of signature $(2n,2n)$ on
 $TM \oplus T^*M$  and such that the quadratic form $\langle {\mathcal J}_1 \cdot , {\mathcal J}_2 \cdot \rangle$ is positive definite on $TM \oplus T^*M$.

By  \cite{AG, Gua}  it turns out that a generalized K\"ahler structure  on $M$   is equivalent to a pair of Hermitian structures 
$(J_+, g) $ and  $(J_-, g),$  where $J_{\pm}$  are two integrable almost complex structures on $M$ and $g$  is a Hermitian metric with respect to $J_{\pm}$, such that the 3-form
 $$
H = d^c_+ \omega_+  = -  d^c_-  \omega_-$$ is closed, where   $\omega_{\pm} (\cdot , \cdot) = g(J_{\pm} \cdot , \cdot)$ are  the    fundamental 2-forms associated with the Hermitian structures $(J_{\pm},g)$  and  $d^c_{\pm} = i( \overline \partial_{\pm}  - \partial_{\pm})$ are the  operators associated with the complex structures  $J_{\pm}$.  
In particular, any K\"ahler metric  $g$ on a complex manifold $(M, J)$  gives rise to a trivial generalized K\"ahler structure by taking $J_+ = J$ and $J_- =  \pm J$.

In the context of Hermitian geometry, the closed 3-form $H$  is also  called the torsion of the generalized K\"ahler structure and it can be  identified with the torsion of the Bismut (or Strominger) connection associated with the Hermitian structure $(J_{\pm},g)$  (see \cite{Bismut,Gau}). A Hermitian  structure $(J, g)$   whose fundamental form $\omega$  is $\partial \overline \partial$-closed is  called  \emph{strong K\"ahler with torsion} (shortly SKT)  or  \emph{pluriclosed}, so  a generalized K\"ahler manifold $(M,J_+,J_-, g)$ consists  of a pair of SKT structures $(J_+, g, \omega_+)$ and  $(J_-, g, \omega_-)$   with opposite Bismut torsion $3$-form.

Hitchin \cite{Hit}  proved  that if  a complex manifold $(M,J) $ admits  a generalized K\"ahler structure $(J_+,  J_-, g, H)$ such that $J = J_+$ and  $J_+, J_-$  do not commute, then the commutator defines a holomorphic Poisson structure  $\pi = [J_+,J_-]g^{-1}$  on $(M,J)$. In this  case  the generalized K\"ahler structure is called {\em non-split}.   If the complex structures $J_+$ and $J_-$  commute,  the generalized K\"ahler structure is said to be {\em{split}} since    $Q = J_+ J_-$  is an involution of the tangent bundle $TM$ and   one has the  splitting $TM = T_+M  \oplus T_- M$  as a direct sum of the  $(\pm 1)$-eigenspaces of $Q$ \cite{AG}.

There are many explicit constructions of non-trivial generalized K\"ahler structures, e.g.\   \cite{AD, AGG, AG, BM,  BCG,  CG2, DM,  FT, Hit}. In particular,  a non-K\"ahler  compact  example    is given  by a  six-dimensional     solvmanifold, i.e., a compact quotient  of  a solvable Lie group by  a uniform discrete subgroup,   endowed with a non-trivial invariant generalized K\"ahler structure  \cite{FT}. This is  in contrast with the case of (compact) nilmanifolds which cannot admit any invariant generalized K\"ahler structures unless they are tori \cite{Cav}. Nevertheless, all six-dimensional nilmanifolds admit invariant generalized complex structures \cite{CG}.

By  \cite{Has}  a solvmanifold has a K\"ahler structure if and only if it is covered by a complex torus which has a structure of  complex torus bundle over a complex torus. No general restrictions  on the existence of generalized K\"ahler structures are known in the case of  compact solvmanifolds.  

The only known examples of (non-K\"ahler) solvable Lie  groups admitting left-invariant  generalized K\"ahler structures are almost abelian \cite{AL,FT}. 
Recall that a connected Lie group $G$ is called \emph{almost abelian} if its Lie algebra $\mathfrak{g}$ admits a codimension-one abelian ideal. In this paper $G$ is always assumed to be connected and simply connected as well.  A    characterization of left-invariant SKT structures on almost abelian Lie groups of any dimension was  obtained in \cite{AL}, but  in real dimension six no classification result is known even for the existence of left-invariant complex structures. Recently,    it  has been   shown that  using almost abelian Lie groups  it is also possible to construct  compact  examples of SKT manifolds whose Bismut connection is  K\"ahler-like  \cite{FTa,ZZ}.

In this paper we first classify, up to isomorphism, six-dimensional  almost abelian Lie groups  admitting  left-invariant complex structures (Theorem \ref{CPX}). This classification can be useful to study also other  types  of Hermitian metrics. We then classify, up to isomorphism, six-dimensional  (non-K\"ahler) almost abelian Lie groups admitting left-invariant SKT structures.  In particular, we prove  that there exist only two families of six-dimensional indecomposable unimodular SKT non-nilpotent almost abelian Lie algebras (Theorem \ref{th_SKT}). 
One  of these Lie algebras corresponds to the example of compact solvmanifold  constructed  in \cite{FT}. Moreover, some of the Lie groups corresponding to the other family of Lie algebras admit compact quotients, too. We  also  discuss  some  results highlighting the differences with the nilpotent case.

Using the characterization  in  \cite{AG, Hit} for split and non-split generalized K\"ahler structures and studying the existence of holomorphic Poisson structures,   we establish which six-dimensional almost abelian Lie groups  have  left-invariant  generalized K\"ahler structures (Theorems \ref{th_GenKahler} and   \ref{th_SGK}). In particular, we show that a six-dimensional unimodular (non-K\"ahler) SKT non-nilpotent almost abelian Lie algebra admitting  holomorphic Poisson structures has to be decomposable and we prove that all left-invariant generalized K\"ahler structures on (non-K\"ahler) six-dimensional almost abelian Lie groups have to be split.  We provide new examples of non-K\"ahler compact solvmanifolds admitting invariant generalized K\"ahler structures. These, together with the example constructed in \cite{FT}, determine the six-dimensional compact almost abelian solvmanifolds admitting invariant generalized K\"ahler structures.
Finally, we study the behavior of the generalized K\"ahler structures  on six-dimensional almost abelian  Lie groups  under the pluriclosed  flow introduced by Streets  and Tian  in \cite{Str, ST, ST1}  and developed in \cite{AL} for almost abelian Lie groups.

The paper is structured as follows: in Section \ref{sec_prelim} we review some  known facts  about generalized K\"ahler structures. Section \ref{sec_SKT}  contains the classification of six-dimensional non-nilpotent almost abelian Lie groups admitting a left-invariant complex structure  and the classification of SKT non-nilpotent almost abelian Lie groups.
Section \ref{sec_HolP} is devoted to    the description  of six-dimensional SKT almost abelian Lie groups  whose complex structure admits non-trivial holomorphic Poisson structures and to the classification of six-dimensional almost abelian Lie groups admitting left-invariant generalized K\"ahler structures. Finally, in Section \ref{sec_flow} we  analyze the behavior of the left-invariant generalized K\"ahler structures under the pluriclosed flow, showing that they are expanding solitons.

{\it Acknowledgements.} The authors would like to thank  Ramiro Lafuente and Luigi Vezzoni  for useful discussions and Jeffrey  Streets for pointing out the reference \cite{Str0}. The authors are also grateful to an anonymous referee for useful comments. The paper is supported by Project PRIN 2017 \lq \lq Real and complex manifolds: Topology, Geometry and Holomorphic Dynamics" and by GNSAGA of INdAM.

\section{Preliminaries on Generalized K\"ahler geometry} \label{sec_prelim}

Generalized geometry deals with structures on the generalized tangent bundle $\mathbb{T}M=TM \oplus T^*M$ of a smooth manifold $M$ of dimension $2n$.

Following \cite{Gua}, $\mathbb{T}M$ can be equipped with a natural inner product $\left< \cdot,\cdot\right>$ of signature $(2n,2n)$,
\[
\left< X + \xi, Y + \eta \right> \coloneqq \frac{1}{2} \left( \eta(X) + \xi(Y) \right),
\]
and, after fixing a closed $3$-form $H$ on $M$, with a bracket operation $[\cdot,\cdot]_H$ called \emph{Courant bracket}
\[
[X+\xi,Y+\eta]_H = [X,Y] + \mathcal{L}_X \eta - \mathcal{L}_Y \xi - \frac{1}{2}d(\eta(X)-\xi(Y)) + \iota_Y \iota_X H, \quad  X+\xi,Y+\eta \in \Gamma(\mathbb{T}M).
\]
The   Courant  bracket  is  said to be \emph{$H$-twisted}  if $H \neq 0$ and \emph{untwisted} if  $H=0$.

Fixing a closed $3$-form $H$ on $M$, a \emph{generalized complex structure} on the pair $(M,H)$ is an almost complex structure $\mathcal{J}$ on $\mathbb{T}M$, i.e., $\mathcal{J} \in \Gamma(\mathbb{T}^*M \otimes \mathbb{T}M)$, $\mathcal{J}^2=-\text{Id}_{\mathbb{T}M}$, which is orthogonal with respect to the inner product $\left< \cdot,\cdot \right>$ and whose $i$-eigenbundle inside $\mathbb{T}M \otimes \C$ is involutive with respect to the $H$-twisted Courant bracket.

For the untwisted case, basic examples of generalized complex structures are provided by classical complex structures $J$ and symplectic structures $\omega$ (namely, non-degenerate closed $2$-forms), interpreted as automorphisms of $\mathbb{T}M$ in matrix form as
\[
\mathcal{J}_J=\begin{pmatrix} -J & 0 \\ 0 & J^* \end{pmatrix},\quad \mathcal{J}_\omega = \begin{pmatrix} 0 & -\omega^{-1} \\ \omega & 0 \end{pmatrix},
\]
respectively. See \cite[Examples 4.20, 4.21]{Gua} for details.

It is possible to associate with  every generalized complex structure $\mathcal{J}$ on $M$ a complex line subbundle $U_{L_\mathcal{J}}$ of the complexified exterior bundle $\Lambda T^*M \otimes \C$, where $L_\mathcal{J}$ denotes the $i$-eigenbundle with respect to $\mathcal{J}$ inside $\mathbb{T}M \otimes \C$. Then $U_{L_\mathcal{J}}$ is by definition the annihilator of $L_\mathcal{J}$ with respect to the spinorial action on complex differential forms, namely
\[
U_{L_\mathcal{J}} = \{ \varphi \in \Lambda T^*M \otimes \C,\, \iota_X \varphi + \xi \wedge \varphi = 0 \text{ for all $X+\xi \in L_\mathcal{J}$}\}.
\] 
The bundle $U_{L_\mathcal{J}}$ takes the name of \emph{canonical bundle} associated with $\mathcal{J}$. We say that $U_{L_\mathcal{J}}$ is \emph{holomorphically trivial} if there exists a nowhere-vanishing section of $U_{L_\mathcal{J}}$ which is closed with respect to the twisted de Rham differential $d-H \wedge$, where the closed $3$-form $H$ corresponds to the twist with respect to which $\mathcal{J}$ is integrable.

A \emph{generalized Riemannian metric} on $M$ is the choice of a $\mathbb{T}M$-subbundle of rank $2n$ on which the inner product is positive-definite. Denoting this subbundle by $E_+$ and its orthogonal complement by $E_-$, one can define the associated involutive automorphism of $\mathbb{T}M$ $\mathcal{G} \coloneqq \text{Id}_{E_+}-\text{Id}_{E_-}$, so that the induced inner product on $\mathbb{T}M$, denoted again by $\mathcal{G}$,
\[
\mathcal{G}(z_1,z_2) \coloneqq \left<\mathcal{G}z_1,z_2\right>, \quad z_1,z_2 \in \Gamma(\mathbb{T}M),
\]
is positive definite.

By \cite[Section 6.2]{Gua}, a generalized Riemannian metric $\mathcal{G}$, viewed as  an automorphism of $\mathbb{T}M$, is always of the form
\[
\mathcal{G}=e^{B}\begin{pmatrix} 0 & g^{-1} \\ g & 0  \end{pmatrix} e^{-B},
\]
for some Riemannian metric  $g$ and  $2$-form $B$ on $M$, where  $e^B$  denotes  the \emph{$B$-field transformation}
\[
e^B = \begin{pmatrix} 1 & 0 \\ B & 1 \end{pmatrix}
\]
and the   functions $g$ and $B$   are   defined by
\[
g(X)(\cdot) \coloneqq g(X,\cdot), \quad B(X)(\cdot) \coloneqq B(X,\cdot), \quad X \in \Gamma (TM).
\]
Note that the map $g^{-1}$ exists by the non-degeneracy of $g$.

\begin{definition} {\normalfont (\cite{Gua})}
A \emph{generalized K\"ahler structure} on $M$ is a pair of commuting generalized complex structures $(\mathcal{J}_1,\mathcal{J}_2)$ on $M$ which are integrable with respect to the same $H$-twisted  Courant bracket and  such that $\mathcal{G}=-\mathcal{J}_1\mathcal{J}_2$ is a generalized Riemannian metric on $M$.
\end{definition}

Actually, a generalized K\"ahler structure can be  restated in terms of Hermitian geometry in the following way: by \cite{Gua, AG}, it is equivalent to a bi-Hermitian structure $(J_+,J_-,g)$, where $J_\pm$ are two complex structures and $g$ is a Hermitian metric with respect to both $J_+$ and $J_-$, satisfying
\[
d^c_+ \omega_+ + d^c_- \omega_- =0,\quad dd^c_+\omega_+ = dd^c_- \omega_- = 0,
\]
where $\omega_\pm(\cdot,\cdot) = g(J_\pm \cdot, \cdot)$ and $d^c_{\pm}=-J_\pm d J_\pm$. In more refined terms, a generalized K\"ahler structure is therefore equivalent to a triple $(J_+,J_-,g),$ where $(J_\pm,g)$ are SKT structures with opposite Bismut torsion $3$-form. Recall that  an SKT structure  is a Hermitian structure whose fundamental form is $dd^c$-closed, or equivalently $\partial \overline \partial$-closed.

In this light, it is clear that trivial examples of generalized K\"ahler structures are provided by genuine K\"ahler structures $(J,g)$, by setting $J_+=J$ and $J_-=\pm J$.

In general, a generalized K\"ahler structure is said to be \emph{split} when the two complex structures $J_\pm$ commute, i.e., $[J_+,J_-]=0$: the name comes from the fact that, in this case, $Q \coloneqq J_+J_-$ defines an involution of $TM$ inducing the splitting $TM=T_+M \oplus T_-M$ as a direct sum of the $(\pm 1)$-eigenbundles with respect to $Q$ (see \cite{AG}). 

When the generalized K\"ahler structure $(J_+,J_-,g)$ is non-split, that is, $J_{\pm}$  do not commute, we still have strong restrictions on the behavior of $[J_+,J_-]$.
To proceed, we need to recall the definition of holomorphic Poisson structure.

In general,  given a complex manifold $(M, J),$  the  complex structure  $J$ determines the \emph{Cauchy-Riemann operator} (see \cite{Gau})
\[
\delb \colon \Gamma(T^{1,0}M) \to \Gamma((T^{0,1}M)^* \otimes T^{1,0}M), 
\]
defined by
\[
\delb_XY \coloneqq [X,Y]^{1,0}, \quad X \in \Gamma(T^{0,1}M),\, Y \in \Gamma(T^{1,0}M),
\]
where  $T^{1,0}M$ and $T^{0,1}M$ denote the $(\pm i)$-eigenbundles of $J$, and   $(\, \cdot \, )^{1,0}$ is the projection from $TM \otimes \C$ onto  $T^{1,0}M$. This extends to an operator on $T^{2,0}M=\Lambda^2T^{1,0}M$ by means of
\[
\delb_X (Y \wedge Z) \coloneqq \delb_XY \wedge Z + Y \wedge \delb_X Z.
\]

Another fundamental operator is the \emph{Schouten bracket}, extending the bracket of vector fields to a bracket for sections of $\Lambda^p TM$, for all $p$. We are interested in the case $p=2$, so that we have
\begin{equation} \label{schoutenbracket}
[X_0 \wedge X_1, Y_0 \wedge Y_1] = \sum_{j,k=0}^1 (-1)^{j+k} [X_j,Y_k] \wedge X_{j+1} \wedge Y_{k+1},
\end{equation}
where the indices in the summation are meant mod $2$ and $X_j,Y_j \in \Gamma(TM)$, $j=0,1$.

\begin{definition}
A \emph{holomorphic Poisson structure} on a complex manifold $(M,J)$ is provided by a $(2,0)$-vector field $\pi \in \Gamma(T^{2,0}M)$ which is both \emph{holomorphic} and \emph{Poisson}, namely
\[
\delb \pi =0,\quad [\pi,\pi]=0.
\]
\end{definition}

Now, let $(J_+,J_-,g)$ be a generalized K\"ahler structure on $M$ and consider the commutator $[J_+,J_-] \in \Gamma(T^*M \otimes TM)$. Applying the inverse of the metric $g$ one gets a bivector $[J_+,J_-]g^{-1} \in \Gamma(\Lambda^2TM)$ which is of type $(2,0)+(0,2)$ with respect to both $J_+$ and $J_-$. It was proven in \cite[Proposition 5]{Hit} that its $(2,0)$-part with respect to $J_+$ (resp. $J_-$) defines a holomorphic Poisson structure with respect to $J_+$ (resp. $J_-$).

\section{Classification of six-dimensional SKT almost abelian Lie groups} \label{sec_SKT}

A characterization of SKT almost abelian Lie groups  in any dimension was obtained in \cite{AL} and  a classification, up to isomorphism,  of  six-dimensional  simply connected almost abelian Lie groups     was given  in \cite{Mu3, Sha} (see  also Tables  \ref{table-indecomp} and  \ref{table-decomp} in the Appendix).  Note that we shall follow the notation given in \cite{Mu3, Sha} to name the associated  Lie algebras; for
instance, the notation $\mathfrak{g}_2 \oplus 4\R=\left(f^{16},0,0,0,0,0\right)$ means that $\mathfrak{g}_2 \oplus 4 \R$
is the (decomposable) Lie algebra determined by a basis  of 1-forms $\{ f^1, \ldots f^6 \}$  such that $df^1= f^1 \wedge f^6$, $d f^j =0$, $j = 2, \ldots, 6$.

In this section we  first  classify, up to isomorphism,  six-dimensional simply connected  almost abelian  Lie groups admitting a left-invariant complex structure  and then establish which of those admit a left-invariant SKT structure. Note that, using the ``symmetrization'' process described in \cite{Bel, FG, Uga}, the existence of an SKT metric on a compact solvmanifold $\Gamma \backslash G$ implies the existence of an invariant one, so in this context the assumption of left-invariance is not restrictive.

Let $G$ be a  $2n$-dimensional  simply connected almost abelian Lie group, i.e., such that  its Lie algebra  $\mathfrak{g}$  has a codimension-one abelian ideal $\mathfrak{h}$. In particular, notice that  $\mathfrak{g}$ has to be  solvable. 
Choosing a basis $\left\{ e_1,\ldots,e_{2n}\right\}$ for $\mathfrak{g}$ such that $\mathfrak{h}=\text{span}\left<e_1,\ldots,e_{2n -1}\right>$, then $\text{ad}_{e_{2n}}$ leaves $\mathfrak{h}$ invariant. The whole Lie algebra structure of $\mathfrak{g}$ is determined by  the derivation $\text{ad}_{e_{2n}} \vert_{\mathfrak{h}}$, allowing to identify $\mathfrak{g}$ with the semidirect product $\R \ltimes_{\text{ad}_{e_{2n}} \vert_{\mathfrak{h}}} \R^{2n-1}$.

A  left-invariant almost Hermitian structure on $G$ is induced  by an almost Hermitian structure $(J,g)$  on the Lie algebra of $ \mathfrak{g}$, where  $J$ is an almost complex structure of $\mathfrak{g}$ and  $g$ is a inner product compatible with $J$. Denote by $\mathfrak{k}\coloneqq \mathfrak{h}^{\perp_g} \cong \mathfrak{g} \slash \mathfrak{h}$ the  $1$-dimensional  orthogonal complement of  $\mathfrak{h}$ in $\mathfrak{g}$ with respect to $g$.  Then $J \mathfrak{k} \subset \mathfrak{h}$, since $J$ is orthogonal, and we can denote $\mathfrak{h}_1 \coloneqq (\mathfrak{k} \oplus J\mathfrak{k})^{\perp_g}$.  Again by orthogonality of $J$, $\mathfrak{h}_1$ must be $J$-invariant, so that we can denote $J_1 \coloneqq J\rvert_{\mathfrak{h}_1}$. 

One is then free to consider an orthonormal basis $\{e_1,\ldots,e_{2n}\}$ of $\mathfrak{g}$ \emph{adapted} to the splitting $\mathfrak{g} = J\mathfrak{k} \oplus \mathfrak{h}_1 \oplus \mathfrak{k}$, i.e., such that 
\[
\mathfrak{k}=\text{span}\left<e_{2n}\right>,\quad \mathfrak{h}_1=\text{span}\left<e_2,\ldots,e_{2n-1}\right>,\quad Je_1=e_{2n}.
\]
With respect to such a basis, the $(2n -1)  \times  (2n - 1)$ matrix $B$ associated with $\text{ad}_{e_{2n}}\rvert_{\mathfrak{h}}$ is of the form
\[
B=\begin{pmatrix} a & w^t \\ v & A \end{pmatrix},
\]
for some $a \in \R$, $v,w \in \mathfrak{h}_1$, $A \in \mathfrak{gl}(\mathfrak{h}_1)$. As shown in \cite{AL}, the almost Hermitian structure $(J,g)$ is thus fully characterized by the algebraic data $(a,v,w,A)$.

If the complex structure $J$ is integrable,  $\mathfrak{h}_1$ must be $\operatorname{ad} \mathfrak{k}$-invariant and the $\operatorname{ad} \mathfrak{k}$-action on $\mathfrak{h}_1$ must commute with $J_1$:
\begin{lemma}{\normalfont (\cite{LR})}
$(J,g)$ is Hermitian if and only if $w=0$ and $[A,J_1]=0$.
\end{lemma}

From now on we assume that the structure $( J,g)$  is Hermitian, so that the matrix $B$ associated with $\text{ad}_{e_{2n}}\rvert_\mathfrak{h}$, with respect to the orthonormal basis $\{e_1,\ldots,e_{2n}\}$,  is of the form
\begin{equation} \label{matrix_Hermitian}
B=\begin{pmatrix} a & 0 \\ v & A \end{pmatrix},
\end{equation}
where $a \in \R$, $v \in \mathfrak{h}_1$, $A \in \mathfrak{gl}(\mathfrak{h}_1)$, $[A,J_1]=0$. The algebraic data reduces to the triple $(a,v,A)$. The Lie algebra  determined by this data will be denoted by $\mathfrak{g}(a,v,A)$.

The classification of six-dimensional nilpotent Lie groups admitting a left-invariant complex structure was obtained in \cite{S}: in particular, the Lie algebra of a six-dimensional nilpotent almost abelian Lie group admitting a left-invariant complex structure has to be isomorphic to one among
\begin{itemize}
\item[]  $(0,0,0,0,0,f^{12})$,
\item[] $(0,0,0,0,f^{12},f^{13})$,
\item[] $(0,0,0,f^{12},f^{13},f^{14})$.
\end{itemize}

Recall that we assume  every almost abelian Lie group $G$ to be connected and simply connected.

\begin{theorem} \label{CPX}
Let $G$ be a six-dimensional non-nilpotent almost abelian Lie group. Then $G$ admits a left-invariant complex structure if and only if its Lie algebra $\mathfrak{g}$ is isomorphic to one of the following:
 \begin{itemize}
\setlength{\itemindent}{-2em}
\item[] $\mathfrak{k}_1^{p,r}=(f^{16},pf^{26},pf^{36},rf^{46},rf^{56},0)$,  \,  $1 \geq |p| \geq |r| > 0$,\smallskip
\item[]
$\mathfrak{k}_{2}^{q,r}=(f^{16},f^{26},qf^{36},rf^{46},rf^{56},0)$, \, $1 > |q| \geq |r| > 0$, \smallskip
\item[]
$\mathfrak{k}_{3}^{q,s}=(f^{16},f^{26},qf^{36},qf^{46},sf^{56},0)$, \, $1 \geq |q| > |s| > 0$, \smallskip
\item[] $\mathfrak{k}_{4}^{q}=(f^{16},f^{26}+f^{36},f^{36},qf^{46},qf^{56},0)$, \, $1 \geq |q| >0$, \smallskip
\item[] $\mathfrak{k}_{5}^{p}=(f^{16},pf^{26}+f^{36},pf^{36},f^{46},pf^{56},0)$, \, $1 > |p| >0$, \smallskip
\item[] $\mathfrak{k}_{6}^{p}=(f^{16},pf^{26}+f^{36},pf^{36},pf^{46}+f^{56},pf^{56},0)$,\smallskip
\item[] $\mathfrak{k}_{7}^{p}=(pf^{16}+f^{26},pf^{26}+f^{36},pf^{36},pf^{46}+f^{56},pf^{56},0)$, \, $p \neq 0$, \smallskip
\item[] $\mathfrak{k}_{8}^{p,q,s}=(pf^{16},qf^{26},qf^{36},sf^{46}+f^{56},-f^{46}+sf^{56},0)$,  \,  $|p| \geq |q|>0$, \smallskip
\item[] $\mathfrak{k}_{9}^{p,r,s}=(pf^{16},pf^{26},rf^{36},sf^{46}+f^{56},-f^{46}+sf^{56},0)$, \, $|p| > |r| > 0$, \smallskip
\item[]  $\mathfrak{k}_{10}^{p,r}=(pf^{16},pf^{26}+f^{36},pf^{36},rf^{46}+f^{56},-f^{46}+rf^{56},0)$, \, $p \neq 0$, \smallskip
\item[] $\mathfrak{k}_{11}^{p,q,r,s}=(pf^{16},qf^{26}+f^{36},-f^{26}+qf^{36},rf^{46}+sf^{56},-sf^{46}+rf^{56},0)$,  \,  $ps \neq 0$,  $(|q| > |r|)$ or
\item[] \hskip 11 cm   $(|q|=|r|,$ $|s| \leq 1)$,\smallskip
\item[] $\mathfrak{k}_{12}^{p,q}=(pf^{16},qf^{26}+f^{36}-f^{46},-f^{26}+qf^{36}-f^{56},qf^{46}+f^{56},-f^{46}+qf^{56},0)$, \,  $p \neq 0$, \smallskip
\item[] $\mathfrak{k}_{13} = (f^{16},0,0,0,0,0)$,\smallskip
\item[] $\mathfrak{k}_{14}=(f^{16},f^{26},0,0,0,0)$, \smallskip
\item[] $\mathfrak{k}_{15}^p=(pf^{16}+f^{26},-f^{16}+pf^{26},0,0,0,0)$, \smallskip
\item[] $\mathfrak{k}_{16}=(f^{16},f^{26}+f^{36},f^{36},0,0,0)$, \smallskip
\item[] $\mathfrak{k}_{17}^p=(f^{16},pf^{26},pf^{36},0,0,0)$, \,  $1 \geq |p| >0$,\smallskip
\item[] $\mathfrak{k}_{18}^q=(f^{16},f^{26},qf^{36},0,0,0)$, \,  $1 > |q| >0$,\smallskip
\item[] $\mathfrak{k}_{19}^{p,q}=(pf^{16},qf^{26}+f^{36},-f^{26}+qf^{36},0,0,0)$,  \, $p \neq 0$, \smallskip
\item[] $\mathfrak{k}_{20}^q=(f^{16},f^{26},qf^{36},qf^{46},0,0)$, \,  $1 \geq |q| > 0$,\smallskip
\item[] $\mathfrak{k}_{21}=(f^{16},f^{26},f^{46},0,0,0)$,\smallskip
\item[] $\mathfrak{k}_{22}^{q,r}=(f^{16},f^{26},qf^{36}+rf^{46},-rf^{36}+qf^{46},0,0)$, \, $r \neq 0$, \smallskip
\item[] $\mathfrak{k}_{23}^p=(pf^{16}+f^{26},-f^{16}+pf^{26},f^{46},0,0,0)$, \smallskip
\item[] $\mathfrak{k}_{24}=(f^{16}+f^{26},f^{26},f^{36}+f^{46},f^{46},0,0)$, \smallskip
\item[] $\mathfrak{k}_{25}^{p,q,r}=(pf^{16}+f^{26},-f^{16}+pf^{26},qf^{36}+rf^{46},-rf^{36}+qf^{46},0,0)$, \, $r \neq 0$, $(|p| > |q|)$ or
\item[]  \hskip 10.3 cm  $(|p|=|q|$, $|r| \leq 1)$, \smallskip
\item[] $\mathfrak{k}_{26}^p=(pf^{16}+f^{26}-f^{36},-f^{16}+pf^{26}-f^{46},pf^{36}+f^{46},-f^{36}+pf^{46},0,0)$.
\end{itemize}
An explicit complex structure,   in terms of the  dual basis $\{f_1, \ldots, f_6\} $,   is given in Table \ref{table-cpx}    for every Lie algebra in the previous list.\end{theorem}

\begin{proof}
Let $J$ be a complex structure on the Lie algebra $\mathfrak{g}$ of $G$. Without loss of generality, one can consider a $J$-Hermitian metric $g$ on $\mathfrak{g}$ and carry out the procedure we have described: let $\{e_1,\ldots,e_6\}$ be an orthonormal basis of $({\mathfrak{g}}, g)$  adapted to the splitting $\mathfrak{g}=J\mathfrak{k} \oplus \mathfrak{h}_1 \oplus \mathfrak{k}$, so that the matrix $B$ associated with $\text{ad}_{e_6}\rvert_{\mathfrak{h}}$ is of the form
\[
B=\begin{pmatrix} a & 0 \\ v & A \end{pmatrix},
\]
with $[A,J_1]=0$. Our discussion will be based on the matrix $A$ and on the interplay between the complex structure  $J$ and the $\operatorname{ad} \mathfrak{k}$-action on $\mathfrak{h}_1$, where $\mathfrak{k}=\text{span}\left<e_6\right>$.

The first step consists into bringing $A$ into a canonical form: depending on its eigenvalues and their multiplicities, there exists a basis $\{e_2,\ldots,e_5\}$ of $\mathfrak{h}_1$ such that, up to rescaling $e_6$, $A$ is represented by a real $4 \times 4$ matrix of one of the following types:
\begin{alignat*}{3}
&A_1=\left( \begin{smallmatrix} p &0&0&0 \\ 0&q&0&0 \\ 0&0&r&0 \\ 0&0&0&s \end{smallmatrix} \right),\quad &  
&A_2=\left( \begin{smallmatrix} p &1&0&0 \\ -1&p&0&0 \\ 0&0&q& 0\\ 0&0&0&r \end{smallmatrix} \right),\quad & 
&A_3=\left( \begin{smallmatrix} p &1&0&0 \\ -1&p&0&0 \\ 0&0&q&r \\ 0&0&-r&q \end{smallmatrix} \right), \\[2pt]
&A_4=\left( \begin{smallmatrix} p&1&0&0 \\ 0&p&0&0 \\0&0&q&0\\ 0&0&0&r \end{smallmatrix} \right),\quad & 
&A_5=\left( \begin{smallmatrix} p &1&0&0 \\ 0&p&0&0 \\ 0&0&q&r \\ 0&0&-r&q \end{smallmatrix} \right), \quad & 
&A_6=\left( \begin{smallmatrix} p &1&0&0 \\ 0&p&1&0 \\ 0&0&p&0 \\ 0&0&0&q \end{smallmatrix} \right), \\[2pt]
&A_7=\left( \begin{smallmatrix} p &1&0&0 \\ 0&p&1&0 \\ 0&0&p&1 \\ 0&0&0&p \end{smallmatrix} \right),\quad & 
&A_8=\left( \begin{smallmatrix} p &1&0&0 \\ 0&p&0&0 \\ 0&0&q&1 \\ 0&0&0&q \end{smallmatrix} \right),\quad &
&A_9=\left( \begin{smallmatrix} p &1&-1&0 \\ -1&p&0&-1 \\ 0&0&p&1 \\ 0&0&-1&p \end{smallmatrix} \right),
\end{alignat*}
for some $p,q,r,s \in \R$, assuming the off-diagonal parameters are non-zero to avoid redundancy.

We now need to establish whether, for some value of the parameters, these matrices may commute with some other matrix squaring to $-\text{Id}$, playing the role of $J_1$. The condition $[A,J_1]=0$ forces the complex structure to preserve the isotypic components of the $\operatorname{ad} \mathfrak{k}$-action on $\mathfrak{h}_1$, since it must map each $\operatorname{ad} \mathfrak{k}$-submodule of $\mathfrak{h}_1$ into an equivalent $\operatorname{ad} \mathfrak{k}$-submodule. For this reason, in particular, there cannot exist odd-dimensional isotypic components. Using these arguments, we can readily discard case $A_6$ and conclude that, in cases $A_2$, $A_4$ and $A_5$, both $\text{span}\left<e_2,e_3\right>$ and $\text{span}\left<e_4,e_5\right>$ should be $J_1$-invariant. Moreover, in case $A_1$, we must require $q=p$, $s=r$ (up to reordering) and similarly $r=q$ in cases $A_4$ and $A_2$.

A simple explicit computation shows that a matrix of the form
\[
\begin{pmatrix} p & 1 \\ 0 & p \end{pmatrix}
\]
can never commute with a matrix squaring to $-\text{Id}$, so that, in light of the discussion above, we may discard cases $A_4$ and $A_5$ and impose $q=p$ in case $A_8$. Case $A_7$ can be easily discarded with an analogous computation. All the remaining cases commute with a suitable $J_1$, namely:
\begin{itemize}
\item[] $A_1$ (with $q=p$, $s=r$), $A_2$ (with $r=q$), $A_3$, $A_9$, with $J_1e_2=e_3$, $J_1e_4=e_5$, \smallskip
\item[] $A_8$ (with $q=p$), with $J_1e_2=e_4$, $J_1e_3=e_5$.
\end{itemize}
Returning to the whole matrix $B$, after the change of basis for $\mathfrak{h}_1$ that we described, we have
\[
B=\begin{pmatrix} a & 0 \\ v & A_i \end{pmatrix},
\]
for $a \in \R$, $v=(v_1,v_2,v_3,v_4)^t$ and $A_i$ one among $A_1$ (with $q=p$, $s=r$), $A_2$ (with $r=q$), $A_3$, $A_8$ (with $q=p$), $A_9$. Now, $a$ is clearly a real eigenvalue of $B$ so that, if it is different from all the eigenvalues of $A_i$, a suitable change of basis of $\mathfrak{h}$ allows to get $v=0$.
Instead, if $a$ coincides with some eigenvalue of $A_i$,  one should check whether the dimension of the $a$-eigenspace of $B$ is either one more than the dimension of the $a$-eigenspace of $A$, in which case, as before, we can get $v=0$ up to a change of basis, or equal to it. To see what happens in this case, let $C_a^k$ denote the $k \times k$ Jordan block
\begin{equation} \label{Jordan_block}
C_a^k=\begin{tikzpicture}[baseline=(current bounding box.center)]
\matrix (m) [matrix of math nodes,nodes in empty cells,right delimiter={)},left delimiter={(}, row sep=0.02cm,column sep=0.02cm]{
a  & 1 &   &  \\
  & & & \\
 & & & 1 \\
 & & & a\\
} ;
\draw[loosely dotted] (m-1-1)-- (m-4-4);
\draw[loosely dotted] (m-1-2)-- (m-3-4);
\end{tikzpicture},\quad C_a^1=\left( a \right).
\end{equation}
Choosing one of the Jordan blocks of $A_i$ relative to the eigenvalue $a$, up to a change of basis   of $\mathfrak{h}_1$, $A$ is in block form
\[
A=\begin{pmatrix} C^k_a & 0 \\ 0 & A^\prime \end{pmatrix},
\]
for some $k$ and some $(4-k) \times (4-k)$ matrix $A^\prime$.
Choosing $v$ suitably and up to a change of basis of $\mathfrak{h}$, it is easy to see that $B$ can be brought into the block form
\[
B=\begin{pmatrix} C_a^{k+1} & 0 \\ 0 & A^\prime \end{pmatrix}.
\]
Thanks to this, we can easily see which algebras one can get starting from the possible matrices $A_i$, $i = 1,2,3,8,9$, up to isomorphism, depending on the value of their parameters and on the behavior of the corresponding matrix $B$:
\begin{itemize}
\item[] $A_1$ yields $\mathfrak{k}_1^{p,r}$, $\mathfrak{k}_2^{q,r}$, $\mathfrak{k}_3^{q,s}$, $\mathfrak{k}_{4}^{q}$, $\mathfrak{k}_{5}^{p}$, $\mathfrak{k}_{13}$, $\mathfrak{k}_{14}$, $\mathfrak{k}_{16}$, $\mathfrak{k}_{17}^p$, $\mathfrak{k}_{18}^q$, $\mathfrak{k}_{20}^{q}$ and $\mathfrak{k}_{21}$,
\item[] $A_2$ yields  $\mathfrak{k}_{8}^{p,q,s}$, $\mathfrak{k}_{9}^{p,r,s}$, $\mathfrak{k}_{10}^{p,r}$, $\mathfrak{k}_{15}^p$, $\mathfrak{k}_{19}^{p,q}$, $\mathfrak{k}_{22}^{q,r}$ and $\mathfrak{k}_{23}^p$,
\item[] $A_3$ yields  $\mathfrak{k}_{11}^{p,q,r,s}$ and $\mathfrak{k}_{25}^{p,q,r}$,
\item[] $A_8$ yields $\mathfrak{k}_{6}^{p}$, $\mathfrak{k}_{7}^p$ and $\mathfrak{k}_{24}$,
\item[] $A_9$ yields  $\mathfrak{k}_{12}^{p,q}$ and $\mathfrak{k}_{26}^p$.
\end{itemize}
This means that for any  of the 26 Lie algebras  in  the previous list one can find an isomorphism with an almost abelian Lie algebra $\mathfrak{g}(a,v,A_i)$, for suitable $a \in \R$, $v \in \R^4$,  and suitable parameters in the entries of the  matrices $A_i$,  $i=1,2,3,8,9$. By construction, $\mathfrak{g}(a,v,A_i)$ supports the complex structure
\[
J=\begin{pmatrix} 0&0&-1 \\ 0&J_1&0 \\ 1&0&0 \end{pmatrix},
\]
in  the fixed basis $\{e_1,\ldots,e_6\}$, where $J_1$ was described above. Then, one can simply use this isomorphism to pull back $J$ and obtain a   complex structure on every Lie algebra. We provide an explicit example in Table \ref{table-cpx} in the Appendix.
\end{proof} 

\begin{remark}
Note that  many of the Lie algebras  in  the previous theorem are  depending on one or more continuous parameters. On each Lie algebra, different sets of parameters may yield isomorphic Lie algebras. In this light, some of the restrictions on the parameters allow to reduce this redundancy. Moreover, one could merge two or more of the previous  Lie algebras together: for example, $\mathfrak{k}_1^{p,r}$, $\mathfrak{k}_2^{q,r}$, $\mathfrak{k}_3^{q,s}$, $\mathfrak{k}_{13}$, $\mathfrak{k}_{17}^p$ and $\mathfrak{k}_{18}^q$ can all be recovered from the Lie algebra
\[
(f^{16},pf^{26},pf^{36},rf^{46},rf^{56},0),
\]
allowing the parameters $p$ and $r$ to take any real value. We decided not to do this to separate decomposable and indecomposable Lie algebras and to maintain the parallelism with the classification of six-dimensional almost abelian Lie algebras in \cite{Mu3,Sha}, as we highlight in Table \ref{table-cpx}.
\end{remark}

\begin{remark}
Note that six-dimensional unimodular solvable Lie algebras admitting  complex structures with  a non-zero closed $(3,0)$-form were classified in \cite{FOU}. 
Among them, the only ones which are almost abelian are  $\mathfrak{k}_{20}^{-1}$  and  $\mathfrak{k}_{25}^{p,-p,1}$.
\end{remark}

The characterization of the SKT condition  on  an almost abelian Lie algebra $\mathfrak{g} (a,v,A)$, determined by the data  $(a,v,A)$,  was obtained in Lemma 4.2 and Theorem 4.6 in \cite{AL}.
\begin{theorem} {\normalfont (\cite{AL})}  \label{SKT_bracket}
$( J,g)$  is SKT  if and only if
\[
aA+A^2+A^tA \in \mathfrak{so} ({\mathfrak{h}_1})
\]
or, equivalently, if $A$ is \emph{normal}, namely $[A,A^t]=0$, and the real part of each eigenvalue of $A$ is equal to either $-\frac{a}{2}$ or $0$.
\end{theorem}

In a similar way we  can show the following 
\begin{lemma} \label{Kahler_bracket}
$ (J,g)$ is  K\"ahler   if and only if
     $A \in \mathfrak{so} ({\mathfrak h}_1)$ and $ v=0.$
\end{lemma}
\begin{proof}
Let $\omega(\cdot,\cdot)=g(J\cdot,\cdot)$ be the fundamental form associated with the Hermitian structure $(J, g)$. Then,
\[
d\omega(X,Y,Z)=g([X,Y],JZ)+g([Z,X],JY)+g([Y,Z],JX), \quad X,Y,Z \in \mathfrak{g}.
\]
 Clearly, the above expression vanishes when all three entries lie in the abelian ideal $\mathfrak{h}$, so that we only need to check the value of $d\omega(e_{2n},e_1,Z)$ and $d\omega(e_{2n},Y,Z)$, for $Y,Z \in \mathfrak{h}_1$. First,
\[
d\omega(e_{2n},e_1,Z)=g([e_{2n},e_1],JZ)=g(v,JZ),
\]
which, by  the $J$-invariance of $\mathfrak{h}_1$, vanishes for all $Z \in \mathfrak{h}_1$ if and only if $v=0$. Then,
\begin{align*}
d\omega(e_{2n},Y,Z)&=g([e_{2n},Y],JZ)-g([e_{2n},Z],Y) \\
                   &=g(AY,J_1Z)-g(AZ,J_1Y) \\
                   &=-g((J_1A+A^tJ_1)Y,Z) \\
                   &=-g((A+A^t)JY,Z),
\end{align*}
where we used that $g$ is Hermitian and $[A,J_1]=0$.  Therefore,  $d\omega(e_{2n},Y,Z)$ vanishes for all $Y,Z \in \mathfrak{h}_1$ if and only if $A+A^t=0$, that is, $A \in \mathfrak{so} ({\mathfrak h}_1)$.
\end{proof}

\begin{remark} \label{remark_normal}
We recall the well-known spectral theorem for normal operators: if $A$ is a normal operator on a metric real vector space $(V,g)$, then it is unitarily diagonalizable as an operator on  the complexification $(V \otimes \C, g \otimes \C)$, while there always exists an orthonormal real basis of $V$ such that the matrix associated with $A$ is in block diagonal form, $A=\text{diag}(\lambda_1,\ldots,\lambda_k,D_1,\ldots,D_h)$, where $\lambda_j \in \R$, $j=1,\ldots,k$, and the $D_j$'s, $j=1,\ldots,h$, are $2 \times 2$ blocks of the form
\[
D_j = \begin{pmatrix} a_j & b_j \\ -b_j & a_j \end{pmatrix},
\]
$a_j,b_j \in \R$, $j=1,\ldots,h$.
The eigenvalues of $A$ are thus $\lambda_1,\ldots,\lambda_k$ and $a_j\pm ib_j$, $j=1,\ldots,h$. In particular, notice that a normal operator is skew-symmetric if and only if all its eigenvalues are pure imaginary.
\end{remark}

Six-dimensional nilpotent Lie groups  admitting a left-invariant SKT structure have been classified in \cite{FPS} and it turns out that the only six-dimensional SKT  almost abelian nilpotent Lie algebra is decomposable and  isomorphic to the direct sum of  $3 \R \oplus \mathfrak{h}_3$, where   $\mathfrak{h}_3$ is  the real 3-dimensional Heisenberg algebra. 

We shall now classify six-dimensional  SKT non-nilpotent almost abelian Lie  groups  which do not admit any left-invariant K\"ahler structures.

\begin{theorem} \label{th_SKT}
Let  $G$  be  a non-nilpotent almost abelian   Lie  group of dimension six. Then 
\begin{enumerate}
\item[$(1)$]  $G$ admits a  left-invariant K\"ahler structure if and only if  its Lie algebra $\mathfrak{g}$ is isomorphic to one the following:
\smallskip
\begin{itemize}
\setlength{\itemindent}{-1em}
\item[]  $\mathfrak{k}_{11}^{p,0,0,s}=\left(pf^{16},f^{36},-f^{26},sf^{56},-sf^{46},0\right)$,  \,  $p \neq 0$, $1 \geq |s| > 0$, \smallskip
\item[] $\mathfrak{k}_{13}=\left(f^{16},0,0,0,0,0\right)$, \smallskip
\item[] $\mathfrak{k}_{15}^0=\left(f^{26},-f^{16},0,0,0,0\right)$, \smallskip
\item[] $\mathfrak{k}_{19}^{p,0}= \left(pf^{16},f^{36},-f^{26},0,0,0\right)$, \,  $p \neq 0$, \smallskip
\item[] $\mathfrak{k}_{25}^{0,0,r}=\left(f^{26},-f^{16},rf^{46},-rf^{36},0,0\right)$,  \,  $1 \geq |r| > 0$.
\end{itemize}
Among these,  only $\mathfrak{k}_{15}^0$ and $\mathfrak{k}_{25}^{0,0,r}$ are unimodular. 

\smallskip

\item[$(2)$]  $G$  admits a left-invariant  SKT structure, but it does not admit any  left-invariant K\"ahler structures, if and only if  its Lie algebra $\mathfrak{g}$ is isomorphic to one of the following:
\smallskip
\begin{itemize}
\setlength{\itemindent}{-1em}
\item[]  $\mathfrak{k}_1^{-\frac{1}{2},-\frac{1}{2}}=\left(f^{16},-\frac{1}{2} f^{26}, -\frac{1}{2} f^{36}, -\frac{1}{2} f^{46}, -\frac{1}{2} f^{56},0\right)$, \smallskip
\item[]  $\mathfrak{k}_{8}^{p,-\frac{p}{2},0}=\left(p f^{16}, -\frac{p}{2} f^{26}, -\frac{p}{2} f^{36}, f^{56},-f^{46},0\right)$, \, $p \neq 0$, \smallskip
\item[]  $\mathfrak{k}_{8}^{p,-\frac{p}{2},-\frac{p}{2}}=\left(p f^{16}, -\frac{p}{2} f^{26}, -\frac{p}{2} f^{36},-\frac{p}{2}f^{46} + f^{56},-f^{46}-\frac{p}{2}f^{56},0\right)$, \,  $p \neq 0$, \smallskip
\item[] $\mathfrak{k}_{11}^{p,-\frac{p}{2},0,s}=\left(p f^{16}, -\frac{p}{2}f^{26} +f^{36}, -f^{26}-\frac{p}{2}f^{36},sf^{56},-sf^{46},0\right)$, \,   $p \neq 0$, $1 \geq |s| > 0$, \smallskip
\item[]  $\mathfrak{k}_{11}^{p,-\frac{p}{2},-\frac{p}{2},s}=\left(p f^{16}, -\frac{p}{2}f^{26} +f^{36}, -f^{26}-\frac{p}{2}f^{36},-\frac{p}{2}f^{46}+sf^{56},-sf^{46}-\frac{p}{2}f^{56},0\right)$, \, $p \neq 0$,
\item[]  \hskip 11.7 cm  $1 \geq |s| > 0$, \smallskip
\item[] $\mathfrak{k}_{17}^{-\frac{1}{2}} = \left(f^{16},-\frac{1}{2} f^{26}, -\frac{1}{2} f^{36},0,0,0\right)$, \smallskip
\item[]  $\mathfrak{k}_{19}^{p,-\frac{p}{2}} =\left(pf^{16},-\frac{p}{2} f^{26} + f^{36},- f^{26} -\frac{p}{2} f^{36},0,0,0\right)$, \,  $p \neq 0$,\smallskip
\item[]  $\mathfrak{k}_{23}^0=\left(f^{26},-f^{16},f^{46},0,0,0\right)$.
\end{itemize}
Among these,  only $\mathfrak{k}_{8}^{p,-\frac{p}{2},0}$, $\mathfrak{k}_{11}^{p,-\frac{p}{2},0,s}$, $\mathfrak{k}_{17}^{-\frac{1}{2}}$, $\mathfrak{k}_{19}^{p,-\frac{p}{2}}$ and $\mathfrak{k}_{23}^0$  are unimodular.
\end{enumerate}
\end{theorem}

\begin{proof}
Let us focus  first on the K\"ahler case. By Lemma \ref{Kahler_bracket}, if $(J,g)$ is a K\"ahler structure on $\mathfrak{g}$, we know that, with respect to an orthonormal basis $\{e_1,\ldots,e_6\}$ adapted to the splitting $\mathfrak{g}=J\mathfrak{k} \oplus \mathfrak{h}_1 \oplus \mathfrak{k}$, the matrix $B$ associated with $\text{ad}_{e_6}\rvert_{\mathfrak{h}}$ will be of the form
\[
B=\begin{pmatrix} a&0 \\0&A \end{pmatrix},
\]
for some $a \in \R$, $A \in \mathfrak{so} ({\mathfrak{h}_1})$. By Remark \ref{remark_normal}, up to a change of  the orthonormal basis of $\mathfrak{h}_1$,  the matrix $A$ is of the form
\[
A=\left( \begin{smallmatrix} 0&b&0&0 \\ -b&0&0&0 \\ 0&0&0&c \\ 0&0&-c&0 \end{smallmatrix} \right),
\]
for some $b,c \in \R$.
Up to scaling $e_6$ and reordering the basis of $\mathfrak{h}$ we then get the isomorphism of $\mathfrak{g}$ with one of the five Lie algebras of the statement, depending on the vanishing of $a$, $b$ and/or $c$. Explicitly, a K\"ahler structure on the Lie algebras $\mathfrak{k}_{11}^{p,0,0,s}$, $\mathfrak{k}_{13}$ and  $\mathfrak{k}_{19}^{p,0}$ is given by
\begin{equation} \label{example1}
J=\left(\begin{smallmatrix} 0&0&0&0&0&-1 \\ 0&0&-1&0&0&0 \\ 0&1&0&0&0&0 \\ 0&0&0&0&-1&0 \\ 0&0&0&1&0&0 \\ 1&0&0&0&0&0 \end{smallmatrix}\right),\quad  g=\sum_{i=1}^6 f^i \otimes f^i,
\end{equation}
while on $\mathfrak{k}_{15}^0$ and $\mathfrak{k}_{25}^{0,0,r}$ we have the example
\[
J=\left(\begin{smallmatrix} 0&-1&0&0&0&0 \\ 1&0&0&0&0&0 \\ 0&0&0&-1&0&0 \\ 0&0&1&0&0&0 \\ 0&0&0&0&0&-1 \\ 0&0&0&0&1&0 \end{smallmatrix}\right),\quad  g=\sum_{i=1}^6 f^i \otimes f^i.
\]

If the structure $(J,g)$ is only SKT, then, by Theorem \ref{SKT_bracket}, we know that, if $\{e_1,\ldots,e_6\}$ is an orthonormal basis adapted to the splitting $J \mathfrak{k} \oplus \mathfrak{h}_1 \oplus \mathfrak{k}$, then the matrix $B$ associated with $\text{ad}_{e_6} \rvert_{\mathfrak{h}}$ is of the form
\[
B=\begin{pmatrix} a&0 \\v&A \end{pmatrix},
\]
with $A$ normal and having eigenvalues with real part equal to $0$ or $-\frac{a}{2}$.

By the spectral theorem for normal operators (see Remark \ref{remark_normal}), $A$ is diagonalizable as an endomorphism of $\mathfrak{h}_1 \otimes \C$. Following the proof of Theorem \ref{CPX}, this implies that, if $B$ is not diagonalizable, then its Jordan form can admit only a single non-diagonalizable $2 \times 2$ block $C_a^2$, in the notation of \eqref{Jordan_block}. This can only happen if $a$ is an eigenvalue of $A$, implying $a=0$, ultimately yielding a Lie algebra isomorphic to $\mathfrak{k}_{23}^0$.

We can then proceed by weeding out the algebras of Theorem \ref{CPX} which cannot fulfill the SKT requirements and those which admit K\"ahler structures, which we have already treated. This leaves us exactly with the eight classes of part (2) of the statement.
All these algebras admit an SKT structure: an example on $\mathfrak{k}_{23}^0$ is provided by
\[
J=\left(\begin{smallmatrix} 0&-1&0&0&0&0 \\ 1&0&0&0&0&0 \\ 0&0&0&0&-1&0 \\ 0&0&0&0&0&-1 \\ 0&0&1&0&0&0 \\ 0&0&0&1&0&0 \end{smallmatrix}\right),\quad  g=\sum_{i=1}^6 f^i \otimes f^i.
\]
On the remaining seven classes, an explicit SKT structure is given by \eqref{example1}. \end{proof}

\begin{remark}
Recall (see \cite[Definition 4]{AOUV}) that a Hermitian connection $\nabla$ on a Hermitian manifold $(M,J,g)$ is called \emph{K\"ahler-like} if its curvature 
\[
R^\nabla(X,Y)=[\nabla_X,\nabla_Y]-\nabla_{[X,Y]},\quad X,Y \in \Gamma(TM), 
\]
satisfies the first Bianchi identity
\[
R^\nabla(X,Y)Z + R^\nabla(Y,Z)X + R^\nabla(Z,X)Y=0,\quad X,Y,Z \in \Gamma(TM),
\]
and the type condition
\[
R^\nabla(X,Y)=R^\nabla(JX,JY),\quad X,Y \in \Gamma(TM).
\]
In \cite{FTa} the authors studied this condition for SKT almost abelian Lie groups, obtaining compact examples of almost abelian solvmanifolds whose Bismut connection is K\"ahler-like.

The six-dimensional compact example constructed in \cite[Example 4.5]{FTa} corresponds to the SKT almost abelian Lie algebra $\mathfrak{k}_{23}^0$.
\end{remark}

\begin{remark} \label{SKT_nilp}
We observe that, unlike the nilpotent case (see \cite[Theorem 1.2]{FPS}), given a six-dimensional almost abelian Lie algebra with a complex structure $J$, the SKT condition might be satisfied by only some Hermitian metrics. Take for instance the algebra $\mathfrak{k}_{17}^{-\frac{1}{2}}$ equipped with the complex structure $J$ in \eqref{example1}: the Hermitian metric in \eqref{example1} is SKT, while the Riemannian metric defined by
\[
g=\left( \begin{smallmatrix} 1&0&0&0&0&0 \\ 0&1&0&\frac{1}{2}&0&0 \\ 0&0&1&0&\frac{1}{2}&0 \\ 0&\frac{1}{2}&0&1&0&0 \\ 0&0&\frac{1}{2}&0&1&0 \\ 0&0&0&0&0&1  \end{smallmatrix} \right)
\]
is still Hermitian but does not satisfy the SKT condition, as one can show with a direct computation.
\end{remark}

We can prove that the torsion of the Bismut connection on  a  non-K\"ahler  six-dimensional SKT  almost abelian Lie algebra  $(\mathfrak{g} (a,v, A), J,g)$ cannot be exact.

\begin{proposition} \label{prop_exact}
Let  $(\mathfrak{g} (a,v, A), J,g)$ be   a  six-dimensional SKT non-nilpotent almost abelian Lie algebra   which does not admit K\"ahler structures. Then, the torsion $3$-form $H=d^c\omega$ associated with $(J,g)$ is not exact.
\end{proposition}
\begin{proof}
Fix an orthonormal basis $\{e_1,\ldots,e_6\}$ adapted to the splitting $\mathfrak{g}=J\mathfrak{k} \oplus \mathfrak{h}_1 \oplus \mathfrak{k}$ and such that $Je_1=e_6$, $Je_2=e_5$, $Je_3=e_4$. Then we know the matrix $B$ associated with $\text{ad}_{e_6}\rvert_{\mathfrak{h}}$ is of the form
\[
A=\left( \begin{smallmatrix} a&0&0&0&0 \\
v_1&A_{11}&A_{12}&A_{13}&A_{14}\\
v_2&A_{21}&A_{22}&A_{23}&A_{24}\\
v_3&-A_{24}&-A_{23}&A_{22}&A_{21}\\
v_4&-A_{14}&-A_{13}&A_{12}&A_{11} \end{smallmatrix} \right), 
\]
where the symmetries of the $4\times 4$ block $A$ corresponding to $\text{ad}_{e_6}\rvert_{\mathfrak{h}_1}$ are due to the requirement $[A,J_1]=0$. Then, an explicit computation yields
\begin{align*}
H=&(A_{13}-A_{24})(e^{123}-e^{145})-(A_{12}+A_{21})(e^{124}+e^{135})-2A_{22}\,e^{134}-2A_{11}\,e^{125}\\&-v_1\,e^{126}-v_2\,e^{136}-v_3\,e^{146}-v_4\,e^{156}.
\end{align*}
Exact $3$-forms lie in $\Lambda^2 \mathfrak{h}^* \wedge \mathfrak{k}^*$, so if we want $H$ to be exact we get some first restrictions on the entries of $A$, which in particular imply that $A$ is skew-symmetric. We can thus discard all Lie algebras but $\mathfrak{k}_{23}^0$ and the nilpotent $3\R \oplus \mathfrak{h}_3 \cong (0,0,0,0,0,f^{12})$.  In the former case, the eigenvalues of $A$ are necessarily $0$  of multiplicity $2$ and $\pm is$, for some $s \in \R - \{0\}$, so that $\mathfrak{h}_1$ splits into two mutually orthogonal $2$-dimensional $\operatorname{ad} \mathfrak{k}$-modules. By exploiting the spectral theorem for normal operators (Remark \ref{remark_normal}) and the fact that $J$ must preserve the two $\operatorname{ad} \mathfrak{k}$-modules of $\mathfrak{h}_1$, as prescribed by the condition $[A,J_1]=0$, we can then assume, without loss of generality, that $B$ is of the form
\[
\left( \begin{smallmatrix} 0&0&0&0&0 \\ v_1&0&0&0&0 \\ v_2&0&0&s&0 \\ v_3&0&-s&0&0 \\ v_4&0&0&0&0 \end{smallmatrix} \right),
\]
with $s \neq 0$ and $v_1^2+v_4^2 \neq 0$ to ensure that the algebra is not isomorphic to $\mathfrak{k}_{15}^0$. Now,
\[
H=-v_1\,e^{126}-v_2\,e^{136}-v_3\,e^{146}-v_4\,e^{156}.
\] 
Imposing that $H$ is equal to $d\eta$ for some generic $2$-form $\eta=\sum_{j<k} \eta_{jk} e^{jk}$, one obtains that necessarily $\eta_{23}=\eta_{24}=\eta_{35}=\eta_{45}=0$. Then one has $d\eta(e_1,e_2,e_6)=\eta_{25}v_4$, $d\eta(e_1,e_5,e_6)=-\eta_{25}v_1$, which should be equal to $-v_1$ and $-v_4$ respectively: this is only possible if $v_1=v_4=0$, which contradicts our hypothesis.

If $\mathfrak{g} \cong 3\R \oplus \mathfrak{h}_3$, one necessarily has $a=0$, $A=0$, $v \neq 0$, and the claim follows by analogous computations.
\end{proof}

\section{Holomorphic  Poisson structures} \label{sec_HolP}

As remarked in Section \ref{sec_prelim}, holomorphic Poisson structures are a fundamental tool in the study  of generalized K\"ahler structures. For this reason, we focus on  almost abelian Lie groups admitting  left-invariant SKT structures, classifying the ones which also admit non-zero left-invariant holomorphic Poisson structures. For those which do not, we get an immediate obstruction to the existence of non-split generalized K\"ahler structures, while, for those which do admit them, we gain additional information about them. This provides an essential tool in the proof of Theorem \ref{th_GenKahler}.

Let $(J,g)$ be a  left-invariant Hermitian structure on a six-dimensional almost abelian Lie group $G$ and let $\mathfrak{g}$ the Lie algebra of $G$. Then, as in Section \ref{sec_SKT}, if we take a basis $\{e_1,\ldots,e_6\}$ adapted to the splitting $\mathfrak{g}=J \mathfrak{k} \oplus \mathfrak{h}_1 \oplus \mathfrak{k}$, the matrix $B$ corresponding to $\text{ad}_{e_6}\rvert_{\mathfrak{h}}$ will be of the form
\[
B=\begin{pmatrix} a & 0 \\ v & A \end{pmatrix},
\]
for some $a \in \R$, $v=v_1e_2+v_2e_3+v_3e_4+v_4e_5 \in \mathfrak{h}_1$, $A \in \mathfrak{gl}(\mathfrak{h}_1)$ such that  $[A,J_1]=0$, where  $J_1=J\rvert_{\mathfrak{h}_1}$. In what follows, we suppose without loss of generality that
\[
J_1=\left( \begin{smallmatrix} 0&0&0&-1\\0&0&-1&0 \\ 0&1&0&0 \\ 1&0&0&0 \end{smallmatrix} \right),
\]
which means that, in order to have $[A,J_1]=0$, $A$ must be of the form
\[
A=\left( \begin{smallmatrix} 
A_{11}&A_{12}&A_{13}&A_{14}\\
A_{21}&A_{22}&A_{23}&A_{24}\\
-A_{24}&-A_{23}&A_{22}&A_{21}\\
-A_{14}&-A_{13}&A_{12}&A_{11} \end{smallmatrix} \right), 
\]
for some $A_{ij} \in \R$, $i=1,2$, $j=1,\ldots,4$. We denote for simplicity
\[
w_1=A_{11}-iA_{14},\quad w_2=A_{12}-iA_{13},\quad w_3=A_{21}-iA_{24},\quad w_4=A_{22}-iA_{23}
\]
and
\[
\alpha=v_1 + iv_4,\quad \beta=v_2 + i v_3.
\]
At the Lie algebra level we have the splitting
\[
\mathfrak{g} \otimes \C = \mathfrak{g}^{1,0} \oplus \mathfrak{g}^{0,1}
\]
into the  $(\pm i)$-eigenspaces with respect to $J$. A basis of  $\mathfrak{g}^{1,0}$ is given by
\[
Z_1=e_1 - ie_6,\quad Z_2=e_2-ie_5,\quad Z_3=e_3 - i e_4,
\]
while their conjugates $\overline{Z}_i$, $i=1,2,3$, provide a basis of $\mathfrak{g}^{0,1}$.
Extending the Lie bracket of $\mathfrak{g}$ to $\mathfrak{g} \otimes \C$, one can compute
\begin{align*}
&[Z_1,Z_2]=-i\left( w_1Z_2 +w_3Z_3 \right), \\
&[Z_1,Z_3]=-i\left( w_2Z_2 + w_4 Z_3\right), \\
&[Z_1,\overline{Z}_1]=-i\left(aZ_1 +\alpha Z_2 + \beta Z_3 +a \overline{Z}_1 + \overline{\alpha} \overline{Z}_2 + \overline{\beta} \overline{Z}_3 \right), \\
&[Z_2,\overline{Z}_1]=-i\left( w_1Z_2+w_3Z_3\right), \\
&[Z_3,\overline{Z}_1]=-i\left( w_2Z_2+w_4Z_3\right).
\end{align*}
All the other brackets between the $Z_j$'s and/or the $\overline{Z}_j$'s vanish, apart from the ones obtained by conjugating the above expressions or exchanging entries.
Recall that we have
\[
\delb X = \sum_{j=1}^3 \overline{\alpha}^j \otimes \left[\overline{Z}_j,X\right]^{1,0} \quad  \in (\mathfrak{g}^{0,1})^* \otimes \mathfrak{g}^{1,0},
\]
for $X \in \mathfrak{g}^{1,0}$, where $\{\alpha^j\}_{j=1,2,3}$ denotes the basis of $(\mathfrak{g}^{1,0})^*$ dual to $\{Z_j\}_{j=1,2,3}$. Looking at the non-zero brackets, it follows that the image of $\delb$ lies in $\{\overline{\alpha}^1\} \otimes \mathfrak{g}^{1,0}$, so that we can reduce to studying $\delb_{\overline{Z}_1}$, which is an endomorphism of $\mathfrak{g}^{1,0}$: one has
\begin{align*}
\delb_{\overline{Z}_1} Z_1 &= i (a Z_1 + \alpha Z_2 + \beta Z_2),\\
\delb_{\overline{Z}_1} Z_2 &= i (w_1Z_2 + w_3 Z_3),\\
\delb_{\overline{Z}_1} Z_3 &= i (w_2 Z_2 + w_4 Z_3).
\end{align*}
The analogous is true also for the extension of $\delb$ to $\mathfrak{g}^{2,0}=\Lambda^2 \mathfrak{g}^{1,0}$, for which we can take the basis $\{Z_1 \wedge Z_2,Z_1 \wedge Z_3,Z_2 \wedge Z_3\}$. We have
\begin{align*}
\delb_{\overline{Z}_1}(Z_1 \wedge Z_2) &= i\left((a+w_1)Z_1 \wedge Z_2 + w_3 Z_1 \wedge Z_3 - \beta Z_2 \wedge Z_3\right),\\
\delb_{\overline{Z}_1}(Z_1 \wedge Z_3) &= i\left(w_2 Z_1 \wedge Z_2 + (a+w_4)Z_1 \wedge Z_3 + \alpha Z_2 \wedge Z_3\right), \\
\delb_{\overline{Z}_1}(Z_2 \wedge Z_3) &= i(w_1 + w_4)Z_2 \wedge Z_3,
\end{align*}
or, in matrix form,
\begin{equation} \label{delbar}
\delb_{\overline{Z}_1}=i\begin{pmatrix}
a+ w_1 & w_2 & 0 \\
w_3 & a+ w_4 & 0 \\
-\beta & \alpha & w_1 + w_4
\end{pmatrix}.
\end{equation}
Focusing now on the Schouten bracket $[\cdot,\cdot] \colon \mathfrak{g}^{2,0} \times \mathfrak{g}^{2,0} \to \mathfrak{g}^{3,0}$, we first notice that $\mathfrak{g}^{3,0} \cong \mathbb{C}$ via the linear map sending $Z_1 \wedge Z_2 \wedge Z_3$ into $1 \in \C$. This allows to identify the Schouten bracket of $(2,0)$-vectors with a complex-valued symmetric bilinear form on $\mathfrak{g}^{2,0}$: an explicit computation using formula \eqref{schoutenbracket} shows that its associated matrix in the basis $\{Z_1 \wedge Z_2,Z_1 \wedge Z_3,Z_2 \wedge Z_3\}$ is
\begin{equation} \label{schouten}
[\cdot,\cdot]=i\begin{pmatrix}
-2w_3 & w_1 - w_4 & 0 \\
w_1 - w_4 & 2w_2 & 0 \\
0 & 0 & 0
\end{pmatrix}.
\end{equation}
A holomorphic Poisson structure on $(\mathfrak{g},J)$ lies in the kernel of \eqref{delbar} and the isotropic cone of \eqref{schouten}.

\begin{theorem} \label{th_HolP}
Let  $G$  be a six-dimensional non-nilpotent almost abelian Lie group. Then $G$ admits a left-invariant  SKT structure $(J,g)$ with $J$ having non-trivial left-invariant holomorphic Poisson structures if and only if  its Lie algebra $\mathfrak{g}$ is isomorphic to one of the following:
\begin{itemize}
\item[]  $\mathfrak{k}_{11}^{p,0,0,1}=\left( pf^{16},f^{36},-f^{26},f^{56},-f^{46},0\right)$, \,  $p \neq 0$, \smallskip
\item[] $\mathfrak{k}_{13}=\left(f^{16},0,0,0,0,0\right)$, \smallskip
\item[]$\mathfrak{k}_{15}^0= \left(f^{26},-f^{16},0,0,0,0\right)$, \smallskip
\item[] $\mathfrak{k}_{23}^0=\left(f^{26},-f^{16},f^{46},0,0,0\right)$, \smallskip
\item[] $\mathfrak{k}_{25}^{0,0,1}= \left(f^{26},-f^{16},f^{46},-f^{36},0,0\right)$.
\end{itemize}
In particular, if $G$ is unimodular, then its Lie algebra has to be decomposable, being it isomorphic to either $\mathfrak{k}_{15}^0$, $\mathfrak{k}_{23}^0$ or $\mathfrak{k}_{25}^{0,0,1}$.
\end{theorem}

We first introduce a lemma.

\begin{lemma} \label{lemma_HolP}
Let $(\mathfrak{g},J,g)$ be a six-dimensional SKT non-nilpotent almost abelian Lie algebra. If $J$ admits holomorphic Poisson structures and $\mathfrak{h}_1$ splits into two mutually orthogonal $2$-dimensional $\operatorname{ad}\mathfrak{k}$-modules which are $J$-invariant, then $\mathfrak{g}$ is isomorphic to one of the algebras of Theorem \ref{th_HolP}.
\end{lemma}
\begin{proof}
With the notations we have introduced, let $\{e_1,\ldots,e_6\}$ be an orthonormal basis  of $(\mathfrak{g}, g)$ adapted to the splitting $\mathfrak{g}=J\mathfrak{k} \oplus \mathfrak{h}_1 \oplus \mathfrak{k}$ and such that $Je_1=e_6$, $Je_2=e_5$, $Je_3=e_4$.
By the spectral theorem (Remark \ref{remark_normal}) and by the assumption in the lemma, it is easy to see that we can assume that the matrix $B$ corresponding to $\text{ad}_{e_6}\rvert_{\mathfrak{h}}$ is of the form
\[
B=\left( \begin{smallmatrix}
a&0&0&0&0\\
v_1&p&0&0&r \\
v_2&0&q&s&0 \\
v_3&0&-s&q&0 \\
v_4&-r&0&0&p
\end{smallmatrix} \right),
\]
with $p,q \in \{0,-\frac{a}{2}\}$.
The matrices associated with $\delb_{\overline{Z}_1}$ and the Schouten bracket with respect to the induced basis for $\mathfrak{g}^{2,0}$ are given by
\[
\delb_{\overline{Z}_1}=i\begin{pmatrix}
a+ w_1 & 0 & 0 \\
0 & a+ w_4 & 0 \\
-\beta & \alpha & w_1 + w_4
\end{pmatrix},\quad [\cdot,\cdot]=i\begin{pmatrix}
0 & w_1 - w_4 & 0 \\
w_1 - w_4 & 0 & 0 \\
0 & 0 & 0
\end{pmatrix}.
\]
To ensure the existence of holomorphic $(2,0)$-vectors one needs 
\begin{align*}
0=\det \delb_{\overline{Z}_1} &=-i(a+w_1)(a+w_4)(w_1+w_4) \\
                              &=-i(a+p-ir)(a+q-is)(p+q-i(r+s)),
\end{align*}
so that we have three cases:

i) $a+p-ir=0$, that is, $r=0$ and $p=-a$. This implies $p=a=0$, so that
\[
B=\left( \begin{smallmatrix}
0&0&0&0&0& \\
v_1 &0&0&0&0 \\
v_2&0&0&s&0& \\
v_3&0&-s&0&0 \\
v_4&0&0&0&0
\end{smallmatrix}
\right),
\]
where $B$ denotes the matrix associated with $\text{ad}_{e_6}\rvert_{\mathfrak{h}}$ in the fixed basis, as usual.
In order for $\mathfrak{g}$ to be non-nilpotent, $s \neq 0$, so that
$\mathfrak{g}$ is isomorphic to either $\mathfrak{k}_{23}^0$ or $\mathfrak{k}_{15}^0$.
In our basis, holomorphic Poisson structures on $\mathfrak{g}$ exist and they are all multiples of
\[
Z_1 \wedge Z_2 + i \tfrac{\beta}{s} Z_2 \wedge Z_3.
\]

ii) $a+q-is=0$, that is, $s=0$, $q=-a$, which is analogous to the previous case after exchanging $e_2$ with $e_3$ and $e_4$ with $e_5$.

iii) $p+q - i (r+s)=0$, that is, $q=-p$, $s=-r$. The fact that $p$ and $q$ must be equal to either $0$ or $-\frac{a}{2}$ forces $q=p=0$. We then have
\[
B=\left( \begin{smallmatrix}
a&0&0&0&0& \\
v_1 &0&0&0&r \\
v_2&0&0&-r&0& \\
v_3&0&r&0&0 \\
v_4&-r&0&0&0
\end{smallmatrix}
\right),
\]
yielding $\mathfrak{g}$ isomorphic to either $\mathfrak{k}_{11}^{p,0,0,1}$, $\mathfrak{k}_{25}^{0,0,1}$ or $\mathfrak{k}_{13}$,
if $a,r \neq 0$, $a = 0$ and $r \neq 0$ or $a \neq 0$ and $r=0$, respectively ($a=r=0$ would imply $\mathfrak{g}$ nilpotent). In all three cases, holomorphic Poisson structures exist and they are multiples of $Z_2 \wedge Z_3$. 
\end{proof}

We have thus also proven that each of the algebras of the statement of Theorem \ref{th_HolP} admits SKT structures $(J,g)$ with $J$ admitting holomorphic Poisson structures. Moreover, such an SKT structure can be found so that $\mathfrak{h}_1$ splits into two $2$-dimensional $J$-invariant and $\operatorname{ad} \mathfrak{k}$-invariant subspaces.

\begin{proof}[Proof  of Theorem \ref{th_HolP}]
Looking at the remaining algebras in the statement, Lemma \ref{lemma_HolP} allows to  discard the algebras
\begin{alignat*}{3}
&\mathfrak{k}_{8}^{p,-\frac{p}{2},0},\qquad &
&\mathfrak{k}_{8}^{p,-\frac{p}{2},-\frac{p}{2}}, \qquad&
&\mathfrak{k}_{11}^{p,0,0,s}, \\ 
&\mathfrak{k}_{11}^{p,-\frac{p}{2},0,s},\qquad & &\mathfrak{k}_{11}^{p,-\frac{p}{2},-\frac{p}{2},s}, \qquad& &\mathfrak{k}_{17}^{-\frac{1}{2}}, \\
&\mathfrak{k}_{19}^{p,0}, \qquad& &\mathfrak{k}_{19}^{p,-\frac{p}{2}}, \qquad& &\mathfrak{k}_{25}^{0,0,r}, \\
\end{alignat*}
for $r,s \neq 1$. As a matter of fact, given any SKT structure $(J,g)$ on any of them, there must exist an orthonormal basis for $\mathfrak{h}_1$ such that, if $e_6$ is a unit norm generator of $\mathfrak{k}$, the matrix $A$ associated with $\text{ad}_{e_6}\rvert_{\mathfrak{h}_1}$ is respectively of the form
\begin{alignat*}{3}
&\left( \begin{smallmatrix} -\frac{a}{2} &0&0&0 \\ 0&-\frac{a}{2}&0&0 \\ 0&0&0&b \\ 0&0&-b&0 \end{smallmatrix} \right), \quad& 
&\left( \begin{smallmatrix} -\frac{a}{2} &0&0&0 \\ 0&-\frac{a}{2}&0&0 \\ 0&0&-\frac{a}{2}&b \\ 0&0&-b&-\frac{a}{2} \end{smallmatrix} \right), \quad&
&\left( \begin{smallmatrix} 0&b&0&0 \\ -b&0&0&0 \\ 0&0&0&c \\ 0&0&-c&0 \end{smallmatrix} \right), \\
&\left( \begin{smallmatrix} -\frac{a}{2} &b&0&0 \\ -b&-\frac{a}{2}&0&0 \\ 0&0&0&c \\ 0&0&-c&0 \end{smallmatrix} \right), \quad& 
&\left( \begin{smallmatrix} -\frac{a}{2} &b&0&0 \\ -b&-\frac{a}{2}&0&0 \\ 0&0&-\frac{a}{2}&c \\ 0&0&-c&-\frac{a}{2} \end{smallmatrix} \right), \quad&
&\left( \begin{smallmatrix} -\frac{a}{2} &0&0&0 \\ 0&-\frac{a}{2}&0&0 \\ 0&0&0&0 \\ 0&0&0&0 \end{smallmatrix} \right), \\
&\left( \begin{smallmatrix} 0 &b&0&0 \\ -b&0&0&0 \\ 0&0&0&0 \\ 0&0&0&0 \end{smallmatrix} \right), \quad& 
&\left( \begin{smallmatrix} -\frac{a}{2} &b&0& 0\\ -b&-\frac{a}{2}&0&0 \\ 0&0&0&0 \\ 0&0&0&0 \end{smallmatrix} \right), \quad& 
&\left( \begin{smallmatrix} 0&b&0&0 \\ -b&0&0&0 \\0&0&0&c\\ 0&0&-c&0 \end{smallmatrix} \right).
\end{alignat*}
for some $a,b,c \in \R-\{0\}$, $b \neq \pm c$. It is then immediate to see that $\text{span}\left<e_2,e_3\right>$ and $\text{span}\left<e_4,e_5\right>$ are non-equivalent orthogonal $\operatorname{ad}\mathfrak{k}$-modules. The condition $[A,J_1]=0$, $J_1=J\rvert_{\mathfrak{h}_1}$, then forces these two modules to be $J$-invariant. Being these algebras not isomorphic to the ones of Lemma \ref{lemma_HolP}, we conclude that, for any SKT structure $(J,g)$ on them, the corresponding $J$ does not admit holomorphic Poisson structures.

Now, the algebra $\mathfrak{k}_1^{-\frac{1}{2},-\frac{1}{2}}$,
too, does not admit holomorphic Poisson structures: for any basis of $\mathfrak{h}_1$, we have that $A=-\frac{a}{2}\text{Id}$, $a \neq 0$. Thus, every $2$-dimensional $J$-invariant subspace of $\mathfrak{h}_1$ is trivially $\operatorname{ad} \mathfrak{k}$-invariant and Lemma \ref{lemma_HolP} applies.

The only remaining algebra is
$\mathfrak{k}_{11}^{p,-\frac{p}{2},-\frac{p}{2},1}$: let $(J,g)$ be an SKT structure on it. Then, by the spectral theorem (Remark \ref{remark_normal}), with respect to some orthonormal basis $\{e_1,\ldots,e_6\}$ adapted to the splitting $J\mathfrak{k} \oplus \mathfrak{h}_1 \oplus \mathfrak{k}$ we have that the matrix $A$ corresponding to $\text{ad}_{e_6}\rvert_{\mathfrak{h}_1}$ is of the form
\[
A=\left( \begin{smallmatrix} -\frac{a}{2} &r&0&0 \\ -r&-\frac{a}{2}&0&0 \\ 0&0&-\frac{a}{2}&r \\ 0&0&-r&-\frac{a}{2} \end{smallmatrix} \right),
\]
for $a=g([e_6,e_1],e_1) \neq 0$ and $r \neq 0$. Thus $\mathfrak{h}_1 = \text{span}\left<e_2, e_3, e_4, e_5 \right>$ splits into two mutually orthogonal $2$-dimensional equivalent $\operatorname{ad} \mathfrak{k}$-modules $\mathfrak{m}_1=\text{span}\left<e_2,e_3\right>$ and $\mathfrak{m}_2=\text{span}\left<e_4,e_5\right>$. By Lemma \ref{lemma_HolP}, $J_1$ cannot preserve these two modules, but then $\mathfrak{m}_1 \oplus J\mathfrak{m}_1 =\mathfrak{h}_1$ and, replacing $e_4$ and $e_5$ with $e_4^\prime=Je_3$ and $e_5^\prime=Je_2$, respectively, one obtains that the matrix $B$ associated with $\text{ad}_{e_6}\rvert_{\mathfrak{h}}$, with respect to the basis $\{e_1,e_2,e_3,e_4^\prime,e_5^\prime\}$, is of the form
\[
B=\left( \begin{smallmatrix}
a&0&0&0&0 \\
v_1&-\frac{a}{2}&r&0&0 \\
v_2&-r&-\frac{a}{2}&0&0 \\
v_3&0&0&-\frac{a}{2}&-r\\
v_4&0&0&r&-\frac{a}{2} \end{smallmatrix} \right).
\]
We have $Je_1=e_6$, $Je_2=e_5^\prime$, $Je_3=e_4^\prime$, so that we may directly apply the discussion at the beginning of this section, obtaining that the matrix associated with $\delb_{\overline{Z}_1}$ with respect to the induced basis for $(2,0)$-vectors is of the form
\[
\delb_{\overline{Z}_1}=i\begin{pmatrix}
\frac{a}{2} & r & 0 \\
-r & \frac{a}{2} & 0 \\
-\beta & \alpha & -a
\end{pmatrix}.
\]
Then
\[
\det \delb_{\overline{Z}_1}=-ia\left(\tfrac{a^2}{4}+r^2\right)\neq 0,
\]
It thus follows that there are no holomorphic $(2,0)$-vectors, hence no holomorphic Poisson structures. This concludes the proof of the theorem.
\end{proof}

\begin{example}
The three unimodular almost abelian Lie groups of Theorem \ref{th_HolP} admit compact quotients by lattices: the compact solvmanifolds obtained from $\mathfrak{k}_{15}^0$ and $\mathfrak{k}_{25}^{0,0,1}$ are K\"ahler and appear in \cite{Has1}: the former corresponds to the product between a hyperelliptic surface and a $2$-torus, while the latter is described in \cite[Example 4]{Has1} as a natural generalization of hyperelliptic surfaces. A lattice on the group corresponding to $\mathfrak{k}_{23}^0$ is given in \cite[Proposition 7.2.7]{Bock}. Therefore, we have obtained three examples of  compact  solvmanifolds admitting SKT structures  and non-trivial invariant holomorphic Poisson structures.
\end{example}

As we have just proved, not all left-invariant SKT structures $(J,g)$ on a six-dimensional almost abelian Lie group are such that $J$ admits non-trivial holomorphic Poisson structures. This constitutes a radical difference with respect to the six-dimensional nilpotent case, treated in \cite{FPS}. As we have already recalled in Remark \ref{SKT_nilp}, the SKT condition for a left-invariant Hermitian structure $(J,g)$ on a six-dimensional nilpotent Lie group $N$  depends solely on the complex structure: by the characterization of \cite[Theorem 1.2]{FPS}, a left-invariant complex structure $J$ on $N$ is SKT if and only $(\mathfrak{n}^{1,0})^*$  admits a basis $\{\alpha_1,\alpha_2,\alpha_3\}$ such that $d\alpha^1=d\alpha^2=0$ and $d\alpha^3$ satisfies some further conditions. Denoting by $\{Z_1,Z_2,Z_3\}$ its dual basis for $\mathfrak{n}^{1,0}$, we obtain that
\[
[\mathfrak{n}_{\C},\mathfrak{n}_{\C}] \subset \text{span}\left<Z_3,\overline{Z}_3\right> \subset \mathfrak{z}(\mathfrak{n}_{\C}),
\]
where $\mathfrak{n}_{\C}=\mathfrak{n} \otimes \C$  and $\mathfrak{z}(\mathfrak{n}_{\C})$ denotes the center of $\mathfrak{n}_{\C}$. Using these relations, one can easily obtain that $X \wedge Z_3$ is a holomorphic Poisson structure for any $X \in \mathfrak{n}^{1,0}$.

\section{Generalized K\"ahler structures on six-dimensional almost abelian Lie groups}

In this Section we study the existence of  left-invariant generalized K\"ahler structures on $6$-di\-men\-sion\-al almost abelian Lie groups.

We first focus on the non-split case, i.e., on generalized K\"ahler structures $(J_+,J_-,g)$ with $[J_+,J_-]\neq 0$. As recalled in Section \ref{sec_prelim}, such generalized K\"ahler structures give rise to a non-trivial holomorphic Poisson structure with respect to $J_\pm$. Going back to Theorem \ref{th_HolP}, we notice that, if $G$ is a six-dimensional non-nilpotent almost abelian Lie group not admitting left-invariant K\"ahler structures but admitting left-invariant SKT structures with non-trivial left-invariant holomorphic Poisson structures, then its Lie algebra has to be isomorphic to $\mathfrak{k}_{23}^0$. This fact simplifies the proof of our next result:

\begin{theorem} \label{th_GenKahler}
Let $G$ be a six-dimensional almost abelian Lie group  not admitting left-invariant K\"ahler structures. 
Then $G$ does not admit any non-split left-invariant generalized K\"ahler structures.
\end{theorem}
\begin{proof}
The claim is true in the nilpotent case \cite{Cav}, since a nilpotent Lie algebra  does not admit any generalized K\"ahler structures. If $G$ is non-nilpotent and has  a non-split left-invariant generalized K\"ahler structure, then, by Theorem \ref{th_HolP}, its Lie algebra $\mathfrak{g}$ is isomorphic to $\mathfrak{k}_{23}^0$.

We start from a generic SKT structure $(J_+,g)$ on  $\mathfrak{g} \cong \mathfrak{k}_{23}^0$: by the same arguments we used in the proof of Proposition \ref{prop_exact} there exists an orthonormal basis $\{e_1,\ldots,e_6\}$ of $(\mathfrak{g}, g)$  adapted to the splitting $\mathfrak{g}=J_+\mathfrak{k} \oplus \mathfrak{h}_1 \oplus \mathfrak{k}$ such that $J_+e_1=e_6$, $J_+e_2=e_5$, $J_+e_3=e_4$ and the matrix $B$ associated with $\text{ad}_{e_6}\rvert_{\mathfrak{h}}$ is of the form
\[
B=\left( \begin{smallmatrix} 0&0&0&0&0 \\ v_1&0&0&0&0 \\ v_2&0&0&s&0 \\ v_3&0&-s&0&0 \\ v_4&0&0&0&0 \end{smallmatrix} \right),
\]
for some $s \in \R-\{0\}$, $v_i \in \R$, $i=1,\ldots,4$, with $v_1^2+v_4^2\neq 0$. By our previous discussion, we know that holomorphic Poisson structures with respect to $J_+$ form a line in $\mathfrak{g}^{2,0}$ generated by
\[
Z_1 \wedge Z_2 + \tfrac{i(v_2+iv_3)}{s} Z_2 \wedge Z_3,
\]
that is,
\begin{equation} \label{holPoisson}
(e_{12}+e_{56})+\tfrac{v_3}{s}(-e_{23}-e_{45})+\tfrac{v_2}{s}(e_{24}-e_{35})+i\big( e_{26}-e_{15} + \tfrac{v_2}{s}(e_{23}+e_{45})+\tfrac{v_3}{s}(e_{24}-e_{35})\big),
\end{equation}
where $e_{ij}=e_i \wedge e_j$. If we assume there exists a complex structure $J_-$ on $\mathfrak{g}$ such that $(J_+,J_-,g)$ is a generalized K\"ahler structure, then $[J_+,J_-]g^{-1} \in \mathfrak{g}^{2,0}\oplus \mathfrak{g}^{0,2}$ should be equal to a (real) multiple of the real or imaginary part of \eqref{holPoisson}. Exploiting the fact that the basis $\{e_1,\ldots,e_6\}$ is orthonormal we then get that $[J_+,J_-] \in \mathfrak{so}(\mathfrak{g},g)$ should be a multiple of the endomorphism
\[
\phi_1 =\left( \begin{smallmatrix} 
0&s&0&0&0&0\\
-s&0&-v_3&v_2&0&0 \\
0&v_3&0&0&-v_2&0 \\
0&-v_2&0&0&-v_3&0\\
0&0&v_2&v_3&0&s \\
0&0&0&0&-s&0
 \end{smallmatrix}\right) \quad \text{or} \quad 
\phi_2 =\left( \begin{smallmatrix} 
0&0&0&0&-s&0\\
0&0&v_2&v_3&0&s \\
0&-v_2&0&0&-v_3&0 \\
0&-v_3&0&0&v_2&0\\
s&0&v_3&-v_2&0&0 \\
0&-s&0&0&0&0
 \end{smallmatrix} \right).
\]
We proceed in this way: we write the generic skew-symmetric $J_-$ in the fixed orthonormal basis and impose that $[J_+,J_-]$ is a multiple of $\phi_1$ or $\phi_2$. Then we impose $J_-^2=-\text{Id}$, the integrability of $J_-$ and the generalized K\"ahler compatibility condition with $J_+$, $d^c_+\omega_+ + d^c_- \omega_-=0$. Then, one obtains that all these conditions are incompatible, so that, by arbitrariness of $J_+$, no generalized K\"ahler structure exists.

We provide details only for the case where $[J_+,J_-]$ is multiple of $\phi_1$, since for the other one the discussion is analogous.
Recall that the integrability of $J_-$ corresponds to the vanishing of the Nijenhuis tensor $N^{J_-} \in \Lambda^2 \mathfrak{g}^* \otimes \mathfrak{g}$, here regarded as a $(0,3)$-tensor $N^{J_-} \in \Lambda^2\mathfrak{g}^* \otimes \mathfrak{g}^*$ with the aid of the metric $g$, by $N^{J_-}(X,Y,Z)\coloneqq g(N^{J_-}(X,Y),Z)$, $X,Y,Z \in \mathfrak{g}$.
Now, the generic skew-symmetric $J_-$ is of the form $J_-=\sum_{j<k} J_{jk}(e^k \otimes e_j - e^j \otimes e_k)$. We then compute $[J_+,J_-]$ and set $J_{36}=J_{14}$, $J_{46}=-J_{13}$ and $J_{56}=-J_{12}$ to kill the desired entries corresponding to the zeros in $\phi_1$. Then we have $N^{J_-}(e_3,e_4,e_6)=s(J_{13}^2+J_{14}^2)$, which forces $J_{13}=J_{14}=0$.

Now, $N^{J_-}(e_3,e_6,e_4)=-s(J_{34}^2-1)$, together with
\begin{alignat*}{2} &N^{J_-}(e_3,e_6,e_2)=sJ_{23}J_{34}, \quad & &N^{J_-}(e_4,e_6,e_2)=sJ_{24}J_{34}, \\ &N^{J_-}(e_4,e_6,e_5)=-sJ_{34}J_{45},\quad & &N^{J_-}(e_3,e_6,e_5)=-sJ_{34}J_{35},
\end{alignat*} imposes $J_{23}=J_{24}=J_{35}=J_{45}=0$. 

Denoting $H_{\pm}=-J_\pm d \omega_{\pm}$, a computation yields
\[
(H_++H_-)(e_1,e_3,e_6)=-v_2(J_{16}^2J_{34}^2+1),\quad (H_++H_-)(e_1,e_4,e_6)=-v_3(J_{16}^2J_{34}^2+1),
\]
whose vanishing forces $v_2=v_3=0$.

We now assume $J_{12}=0$. Recalling that $v_1^2+v_4^2 \neq 0$, we have that
\[
N^{J_-}(e_1,e_6,e_2)=v_1(J_{16}^2-1),\quad N^{J_-}(e_1,e_6,e_5)=v_4(J_{16}^2-1),
\]
together with
\begin{alignat*}{2} &N^{J_-}(e_1,e_5,e_2)=v_1J_{15}J_{16}, \quad & &N^{J_-}(e_1,e_5,e_5)=v_4J_{15}J_{16}, \\ &N^{J_-}(e_2,e_6,e_2)=v_1J_{26}J_{16}, \quad & &N^{J_-}(e_2,e_6,e_5)=v_4J_{26}J_{16},
\end{alignat*}
imply $J_{15}=J_{26}=0$. At this point,
\[
H_++H_-=-v_1(J_{16}^2J_{25}^2+1)e^{126} - v_4(J_{16}^2J_{25}^2+1)e^{156},
\]
which can never vanish by our hypotheses.

Let us assume $J_{12} \neq 0$, instead. Noticing that
\[
(J_-^2)_{15}=J_{12}(J_{16}+J_{25}),\quad (J_-^2)_{16}=J_{12}(J_{26}-J_{15}),
\]
we must have $J_{25}=-J_{16}$, $J_{26}=J_{15}$, but now the vanishing of
\[
N^{J_-}(e_2,e_5,e_2)=v_1(J_{12}^2+J_{15}^2),\quad N^{J_-}(e_2,e_5,e_5)=v_4(J_{12}^2+J_{15}^2)
\]
produces a contradiction.
\end{proof}

Having discussed the non-split case, we now examine split generalized K\"ahler structures which, we recall, are those whose complex structures $J_+$ and $J_-$ commute.

\begin{theorem} \label{th_SGK}
Let $G$ be six-dimensional almost abelian Lie group.  Then  $G$ admits a left-invariant  split generalized K\"ahler structure,  but no  left-invariant K\"ahler structures, if and only if  its Lie algebra $\mathfrak{g}$ is isomorphic to one of the following:
\begin{itemize}
\setlength{\itemindent}{-1em}
\item[]  $\mathfrak{k}_1^{-\frac{1}{2},-\frac{1}{2}}=(f^{16},-\frac{1}{2} f^{26}, -\frac{1}{2} f^{36}, -\frac{1}{2} f^{46}, -\frac{1}{2} f^{56},0)$, \smallskip
\item[] $\mathfrak{k}_{8}^{p,-\frac{p}{2},0}=(p f^{16}, -\frac{p}{2} f^{26}, -\frac{p}{2} f^{36},f^{56},-f^{46},0)$, \,  $p \neq 0$, \smallskip
\item[]  $\mathfrak{k}_{8}^{p,-\frac{p}{2},-\frac{p}{2}}=(p f^{16}, -\frac{p}{2} f^{26}, -\frac{p}{2} f^{36},-\frac{p}{2}f^{46} + f^{56},-f^{46}-\frac{p}{2}f^{56},0)$, \, $p \neq 0$, \smallskip
\item[] $\mathfrak{k}_{11}^{p,-\frac{p}{2},0,s}=(p f^{16}, -\frac{p}{2}f^{26} +f^{36}, -f^{26}-\frac{p}{2}f^{36},sf^{56},-sf^{46},0)$, \, $p \neq 0$, $1 \geq |s| > 0$, \smallskip
\item[]  $\mathfrak{k}_{11}^{p,-\frac{p}{2},-\frac{p}{2},s}=(p f^{16}, -\frac{p}{2}f^{26} +f^{36}, -f^{26}-\frac{p}{2}f^{36},-\frac{p}{2}f^{46}+sf^{56},-sf^{46}-\frac{p}{2}f^{56},0)$, \, $\text{$p \neq 0$, $1 \geq |s| > 0$}$,\smallskip
\item[] $\mathfrak{k}_{17}^{-\frac{1}{2}}= (f^{16},-\frac{1}{2} f^{26}, -\frac{1}{2} f^{36},0,0,0)$, \smallskip
\item[] $\mathfrak{k}_{19}^{p,-\frac{p}{2}} =(pf^{16},-\frac{p}{2} f^{26} + f^{36}, -f^{26} -\frac{p}{2} f^{36},0,0,0)$, \, $p \neq 0$.
\end{itemize}
Among them, only $\mathfrak{k}_{8}^{p,-\frac{p}{2},0}$, $\mathfrak{k}_{11}^{p,-\frac{p}{2},0,s}$, $\mathfrak{k}_{17}^{-\frac{1}{2}}$ and $\mathfrak{k}_{19}^{p,-\frac{p}{2}}$ are unimodular.
\end{theorem}

\begin{proof}
A necessary condition to admit a generalized K\"ahler structure is the existence of an SKT structure. Since we want $\mathfrak{g}$ to admit no K\"ahler structures, $\mathfrak{g}$ is isomorphic to one of the eight Lie algebras of part  (2)  of Theorem \ref{th_SKT}. 

Moreover,  considering the explicit  SKT structures  that we found  for seven of these Lie algebras (all but $\mathfrak{k}_{23}^0$), the splitting $\mathfrak{g}=J\mathfrak{k} \oplus \mathfrak{h}_1 \oplus \mathfrak{k}$ is such that $J \mathfrak{k}$ is $\operatorname{ad} \mathfrak{k}$-invariant (that is, $v=0$). We may then conclude by \cite[Proposition 4.7]{AL} that a split generalized K\"ahler structure $(J_+,J_-,g)$ on each of those algebras is given by
\begin{equation} \label{GK_example}
J_+=\left( \begin{smallmatrix} 0&0&0&0&0&-1 \\ 0&0&-1&0&0&0 \\ 0&1&0&0&0&0 \\ 0&0&0&0&-1&0 \\ 0&0&0&1&0&0 \\ 1&0&0&0&0&0 \end{smallmatrix} \right), \quad J_-=\left( \begin{smallmatrix} 0&0&0&0&0&-1 \\ 0&0&1&0&0&0 \\ 0&-1&0&0&0&0 \\ 0&0&0&0&1&0 \\ 0&0&0&-1&0&0 \\ 1&0&0&0&0&0 \end{smallmatrix} \right),\quad g=\sum_{i=1}^6 f^i \otimes f^i.
\end{equation}
The corresponding torsion $3$-form $H=d^c_+\omega_+$ is given by:
\begin{align*}
H&=f^{123}, && \hspace{-2em} \text{for $\mathfrak{k}_{17}^{-\frac{1}{2}}$}, \\
H&=p\,f^{123}, && \hspace{-2em} \text{for $\mathfrak{k}_{19}^{p,-\frac{p}{2}}$, $\mathfrak{k}_{8}^{p,-\frac{p}{2},0}$, $\mathfrak{k}_{11}^{p,-\frac{p}{2},0,s}$}, \\
H&=f^{123} + f^{145}, && \hspace{-2em} \text{for $\mathfrak{k}_1^{-\frac{1}{2},-\frac{1}{2}}$},\\
H&=p\,f^{123}+p\,f^{145},  && \hspace{-2em} \text{for $\mathfrak{k}_{8}^{p,-\frac{p}{2},-\frac{p}{2}}$, $\mathfrak{k}_{11}^{p,-\frac{p}{2},-\frac{p}{2},s}$}.
\end{align*}

The remaining algebra $\mathfrak{k}_{23}^0$ can be discussed in the same way as in Theorem \ref{th_GenKahler}, noticing that its proof actually never assumes $[J_+,J_-]$ to be non-vanishing.
\end{proof} 

\begin{remark}
The existence of a split generalized K\"ahler structure on the unimodular algebra $\mathfrak{k}_{8}^{p,-\frac{p}{2},0}$ was first determined in \cite{FT}. Theorem \ref{th_SGK} thus provides new examples of  solvable  Lie  algebras which admit generalized K\"ahler structures but no K\"ahler structures.
\end{remark}

\begin{example}
In \cite{FT}, a non-K\"ahler compact quotient by a lattice was explicitly determined for the Lie group with Lie algebra $\mathfrak{k}_{8}^{-\frac{1}{\pi},\frac{1}{2 \pi},0}$: the resulting compact solvmanifold is the total space of a $2$-torus bundle over an Inoue surface.

By \cite{CM}, for some choices for $p$ and $s$ the Lie group with Lie algebra $\mathfrak{k}_{11}^{p,-\frac{p}{2},0,s}$ admits compact quotients by lattices. The resulting compact solvmanifolds $\Gamma \backslash G$ are non-K\"ahler, since $b_1(\Gamma \backslash G)=3$.

The groups corresponding to the Lie algebras $\mathfrak{k}_{19}^{p,-\frac{p}{2}}$ admit compact quotients, corresponding to products of  an Inoue surface and a $2$-torus \cite[Section 5]{Has1}.

The Lie group corresponding to the decomposable Lie algebra $\mathfrak{k}_{17}^{-\frac{1}{2}}$ cannot admit compact quotients by lattices. In fact (see \cite{Bock}) the associated Lie group has a lattice if and only if there exists a real number  $t_0 \neq 0$ such that the matrix associated with $\text{exp}(t_0\text{ad}_{f_6})$ with respect to the fixed basis $\{f_1,\ldots,f_6\}$ is conjugate to an integer matrix. If this were the case, the characteristic polynomial of $\text{exp}(t_0\text{ad}_{f_6})$, which is of the form $P(x)=(x-1)^3Q(x)$, would be such that $Q(x)$ is an integer polynomial $Q(x)=x^3-kx^2+lx^2-1$ with roots $e^{t_0}$, $e^{-\frac{t_0}{2}}$, $e^{-\frac{t_0}{2}}$. By \cite[Lemma 2.2]{Has2}, this implies $e^{-\frac{t_0}{2}}=1$, i.e., $t_0=0$, a contradiction.

In this way we  determine the six-dimensional non-K\"ahler almost abelian compact solvmanifolds admitting invariant generalized K\"ahler structures.
\end{example}

Recall that, by \cite[Theorem 3.1]{CG}, any left-invariant generalized complex structure on a nilmanifold must have holomorphically trivial canonical bundle. This was exploited in \cite{Cav} to prove that the only compact nilmanifolds admitting generalized K\"ahler structures are tori. Thus, it is natural to check if similar results about the canonical bundles hold in the almost abelian case.

Let $G$ be a six-dimensional almost abelian Lie group equipped with a left-invariant generalized K\"ahler structure $(\mathcal{J}_1,\mathcal{J}_2)$. By left-invariance, the canonical bundles of $\mathcal{J}_1$ and $\mathcal{J}_2$ can be identified with complex lines inside the complexified exterior algebra $\Lambda \mathfrak{g}^* \otimes \C$.

Fixing the twist given by $H=d^c_+\omega_+$ on $\mathfrak{g} \oplus \mathfrak{g}^*$, by \cite{Gua1} the $i$-eigenspaces of the generalized complex structures $\mathcal{J}_1$ and $\mathcal{J}_2$ are given respectively by
\begin{align*}
L_1 &= l_+ \oplus l_- = e^{-i\omega_+}\mathfrak{g}^{1,0}_+ \oplus e^{i\omega_-}\mathfrak{g}^{1,0}_- \subset \left(\mathfrak{g} \oplus \mathfrak{g}^*\right) \otimes \mathbb{C},\\
L_2 &= l_+ \oplus \overline{l_-} =e^{-i\omega_+}\mathfrak{g}^{1,0}_+ \oplus e^{-i\omega_-}\mathfrak{g}^{0,1}_- \subset \left(\mathfrak{g} \oplus \mathfrak{g}^*\right) \otimes \mathbb{C},
\end{align*}
where
\begin{align*}
\mathfrak{g}^{1,0}_\pm = \{ X \in \mathfrak{g} \otimes \mathbb{C},\,J_{\pm}X=iX \}, \quad
\mathfrak{g}^{0,1}_\pm = \{ X \in \mathfrak{g} \otimes \mathbb{C},\,J_{\pm}X=-iX \}
\end{align*} 
and $\omega_\pm(\cdot,\cdot)=g(J_\pm \cdot,\cdot)$ are the fundamental forms associated with $(J_\pm,g)$.

Then,
\[
U_{L_1} = U_{l_+} \cap U_{l_-},\quad U_{L_2} = U_{l_+} \cap \overline{U_{l_-}}
\]
where, by \cite[formula (2.13)]{Gua}, one has
\[
U_{l_+} = e^{i\omega_+} \Lambda (\mathfrak{g}^{0,1}_+ )^*, \quad
U_{l_-} = e^{-i\omega_-} \Lambda (\mathfrak{g}^{0,1}_- )^*,
\]
where
\[
e^B \varphi = \sum_{k=0}^{\infty} \tfrac{1}{k!} B^k \wedge \varphi= \varphi + B \wedge \varphi + \tfrac{1}{2} B \wedge B \wedge \varphi + \ldots
\]

For all the groups of Theorem \ref{th_SKT}, the split generalized K\"ahler structure in \eqref{GK_example} determines
\[
\omega_\pm = f^{16} \pm f^{23} \pm f^{45},\quad (\mathfrak{g}_\pm^{0,1})^*=\text{span}\left<f^1-if^6,f^2 \mp if^3,f^4 \mp i f^5\right>,
\]
so that the canonical bundles $U_{L_1}$ and $U_{L_2}$ are generated by the left-invariant complex differential forms
\begin{align*}
\rho_1&=e^{i\omega_+}(f^1 - if^6)=f^1-if^6 +i\omega_+ \wedge (f^1-if^6) -\tfrac{1}{2}\omega_+ \wedge \omega_+ \wedge (f^1 -if^6),\\
\rho_2&=e^{i\omega_+}(f^2-if^3)\wedge(f^4-if^5) = (f^2-if^3)\wedge(f^4-if^5) + i\omega_+ \wedge(f^2-if^3)\wedge(f^4-if^5),
\end{align*}
respectively, as shown by a direct computation. Recall that $U_{L_i}$ is holomorphically trivial if its generator $\rho_i$ is closed with respect to the twisted exterior differential $d-H \wedge$ determined by the splitting $H=d^c_+ \omega_+$. A simple computation shows that this is never the case in our examples.

\section{Generalized K\"ahler flow} \label{sec_flow}

In \cite{ST, ST2} J. Streets and G. Tian introduced a geometric flow for Hermitian metrics on a complex manifold $M$, preserving the SKT condition and generalizing the K\"ahler-Ricci flow. This flow, which takes the name of \emph{pluriclosed flow}, is expressed through the fundamental forms of the flowing metrics as
\[
\dot{\omega}=-(\rho^B_\omega)^{1,1},\quad \omega(0)=\omega_0,
\]
where $(\rho^B_\omega)^{1,1}$ denotes the $(1,1)$-part of the \emph{Bismut Ricci form} associated with $\omega$, having local expression
\[
\rho^B_\omega(X,Y) = -\frac{1}{2} \sum_{k=1}^{2n} g(R^B(X,Y)e_k,Je_k),
\]
for any local orthonormal frame $\{e_1,\ldots,e_{2n}\}$, $2n=\dim M$, where $R^B$ denotes the curvature of the Bismut connection $\nabla^B$.

Up to time dependent diffeomorphisms, that is, up to a change of \emph{gauge} (see \cite{ST2} for further details), the pluriclosed flow starting from an SKT metric is equivalent to the paired flow for a Riemannian metric and a closed $3$-form (preserving the cohomology class of the latter) defined by
\begin{equation} \label{genricciflow}
\begin{cases}
\dot{g}=-2\,\text{Ric}_g + \frac{1}{2} H \upperset{g}{\circ} H, & g(0)=g_0=\omega_0(\cdot,J\cdot),\\
\dot{H}=-\Delta_g H, & H(0)=H_0=d^c \omega_0,
\end{cases}
\end{equation}
where $\text{Ric}_g$ is the Ricci tensor associated with $g$, $H \upperset{g}{\circ} H$ is given by  \[H \upperset{g}{\circ} H(X,Y)=g(\iota_XH,\iota_YH)\] and $\Delta_g=dd^*_g +d^*_gd$ is the Hodge Laplacian associated with the metric $g$ and the fixed orientation. These equations correspond to the \emph{$B$-field renormalization group flow} of Type II string theory and were recently generalized by Streets \cite{Str0} and Garcia-Fernandez \cite{Gar} to define the \emph{generalized Ricci flow} on Courant algebroids: for example, a solution $(g(t),H(t)=H_0+dB(t))$ to \eqref{genricciflow} can be interpreted in this context as a family of generalized Riemannian metrics
\[
\mathcal{G}(t)=e^{B(t)}\begin{pmatrix} 0 & g(t)^{-1} \\ g(t) & 0 \end{pmatrix} e^{-B(t)}
\]
on the generalized tangent bundle $\mathbb{T}M$ equipped with the $H_0$-twisted Courant  bracket.

Given a split generalized K\"ahler structure $(J_+,J_-,g)$, the pluriclosed flow starting from the SKT structure $(J_+,g)$ produces a family of SKT metrics with respect to both $J_+$ and $J_-$, preserving the generalized K\"ahler condition $d^c_+ \omega_+ + d^c_- \omega_- =0$, so that one may say that the given split generalized K\"ahler structure evolves by $(J_+,J_-,g(t))$. This flow is also called \emph{generalized K\"ahler flow}  (\cite{Str}).

When one works on Lie groups, left-invariant initial conditions yield left-invariant solutions, so that the pluriclosed flow and the generalized K\"ahler flow reduce to systems of \textsc{ode}s on the associated Lie algebra.

We recall  that a SKT structure $(J,g)$ on a real Lie algebra  $\mathfrak{g}$  is a \emph{pluriclosed soliton} if the pluriclosed flow starting from $(J,g)$ evolves simply by rescaling and time-dependent biholomorphisms, namely $g(t)=c(t)\varphi_t^*g$, with $c(t) \in \R$ and $\varphi_t$ biholomorphisms. More precisely, we say that $(J,g)$ is a \emph{shrinking}, \emph{expanding} or \emph{steady} soliton on $\mathfrak{g}$  if $c= c(t)$ is respectively decreasing, increasing or constantly equal to $1$.

Analogously, we say that a split generalized K\"ahler structure $(J_+,J_-,g)$ on $\mathfrak{g}$ is a \emph{soliton} for the generalized K\"ahler flow if $(J_+,J_-,g(t))=(J_+,J_-,c(t)\varphi_t^*g)$.

We now briefly review the \emph{bracket flow} technique applied to the case of the pluriclosed flow, as treated in \cite{EFV1, AL}, to which we refer the reader for further details.

Given a Lie algebra $\mathfrak{g}$, view it as a pair $(\mathfrak{g},\mu_0)$, where $\mathfrak{g}$ denotes the underlying vector space and $\mu_0 \in V(\mathfrak{g}) = \Lambda^2 \mathfrak{g}^* \otimes \mathfrak{g}$ denotes the Lie bracket. Fix then a complex structure $J$ on $(\mathfrak{g},\mu_0)$. The Lie group $\text{GL}(\mathfrak{g},J)$ of automorphisms of $\mathfrak{g}$ preserving $J$ acts transitively on the set of Hermitian metrics with respect to $J$ via pullback, so that, if $g_0$ is a Hermitian metric on $(\mathfrak{g},\mu_0,J)$, the pluriclosed flow starting from $(J,g_0)$ yields a family $(J,h(t)^*g_0)$, for some $h(t) \subset \text{GL}(\mathfrak{g},J)$. One then observes that
\[
h(t) \colon (\mathfrak{g},\mu_0,J,h(t)^*g_0) \to (\mathfrak{g},h(t) \cdot \mu_0,J,g_0)
\]
is an isomorphism of Hermitian structures, namely $h(t)$ is a Lie algebra isomorphism which is orthogonal and biholomorphic. Here we denoted
\[
h \cdot \mu = (h^{-1})^*\mu=h\mu(h^{-1} \cdot, h^{-1} \cdot).
\]
Let $\mu(t)=h(t) \cdot \mu_0$. Then,  up to time-dependent biholomorphisms, the pluriclosed flow of a Hermitian structure  $(J,g_0)$ on  $(\mathfrak{g}, \mu_0)$ can be interpreted as a flow $\mu(t)$ on $V(\mathfrak{g})$, such that $\mu(t) \in \text{GL}(\mathfrak{g},J) \cdot \mu_0$ for all $t$. Denote by $\rho^B_{\omega_0,\mu}$ the Bismut Ricci form associated with the left-invariant extension of $\omega_0$ on  the unique simply connected Lie group $G_{\mu}$ having Lie algebra $(\mathfrak{g},\mu)$, and by $\rho_\mu^B$ its restriction to $\mathfrak{g}$, i.e.,
\[
\rho^B_{\mu} \coloneqq \rho^B_{\omega_0,\mu} \rvert_e \in \Lambda^2 T^*_e G_{\mu} \cong \Lambda^2 \mathfrak{g}^*.
\] 
The  evolution of
$\mu(t)$ is given by the so-called bracket flow
\begin{equation} \label{brflow}
\dot{\mu}= - \pi(P_\mu)\mu,\quad \mu(0)=\mu_0,
\end{equation}
where
\[
P_\mu = \tfrac{1}{2} \omega_0^{-1} (\rho_\mu^B)^{1,1} \in \mathfrak{gl}(\mathfrak{g}), \quad \omega_0(\cdot,\cdot)=g_0(J \cdot, \cdot),
\]
and 
\[
(\pi(A)\mu)(X,Y)=A\mu(X,Y) - \mu(AX,Y) - \mu(X,AY),
\]
for any $A \in \mathfrak{gl}(\mathfrak{g})$, $\mu \in V(\mathfrak{g})$, $X,Y \in \mathfrak{g}$. 
Applying a \emph{gauge} to the bracket flow \eqref{brflow}, namely considering a flow of the form
\begin{equation} \label{gauged_brflow}
\dot{\bar{\mu}}=\pi(P_{\bar{\mu}} - U_{\bar{\mu}})\bar{\mu}, \quad \bar{\mu}(0)=\mu_0,
\end{equation}
for some smooth map $U \colon V(\mathfrak{g}) \to \mathfrak{u}(\mathfrak{g},J)$, then, by \cite[Theorem 2.2]{AL}, for any $\mu_0 \in V(\mathfrak{g})$, there exist $k(t) \subset \text{U}(\mathfrak{g},J)$ such that $\bar{\mu}(t) = k(t) \cdot \mu(t)=k(t)h(t) \cdot \mu_0$ for all $t$, where $\mu(t)$ and $\bar{\mu}(t)$ respectively denote the solutions to \eqref{brflow} and \eqref{gauged_brflow}.

This implies that, given an SKT Lie algebra $(\mathfrak{g},\mu_0,J,g_0)$, assuming there exists a gauged bracket flow such that $\mu_0$ evolves only by rescaling, $\bar{\mu}(t)=c(t) \mu_0$, $c(t) \in \R$, then $(J,g_0)$ is a pluriclosed soliton on $\mathfrak{g}$. The converse holds as well.

It is now natural to study the behaviour of the  split generalized K\"ahler   structures on the Lie algebras in Theorem \ref{th_SGK} under the generalized K\"ahler flow.

To do this, we first recall the setup for the pluriclosed flow of left-invariant SKT structures on almost abelian Lie groups \cite{AL}, in terms of the bracket flow.
Let $G$ be a $2n$-dimensional almost abelian Lie group with Lie algebra $(\mathfrak{g},\mu)$. As we have reviewed in Section \ref{sec_SKT}, given an SKT structure $(J,g)$ on it, there exists a $g$-orthonormal basis $\{e_1,\ldots,e_{2n}\}$ of $\mathfrak{g}$ such that $\mathfrak{h}=\text{span}\left<e_1,\ldots,e_{2n-1}\right>$ and the matrix $B$ associated with $\text{ad}_{e_{2n}}\rvert_{\mathfrak{h}}$ is of the form \eqref{matrix_Hermitian}.
In general, the bracket flow \eqref{brflow} will not preserve this form. In order to adjust this, in \cite{AL} the authors introduced a gauged bracket flow of the form \eqref{gauged_brflow}, which instead preserves the nilradical $\mathfrak{h}$, so that the pluriclosed flow is equivalent to a system of \textsc{ode}s for the triple $(a,v,A) \in \R \times \R^{2n-2} \times \R^{2n-2,2n-2}$, namely
\[
\begin{cases}
\dot{a}=ca, & a(0)=a_0,\\
\dot{v}=cv+Sv-\tfrac{1}{2} \lVert v \rVert^2v, & v(0)=v_0,\\
\dot{A}=cA, & A(0)=A_0,
\end{cases}
\]
where $c=-\left( \tfrac{k}{4} + \tfrac{1}{2} \right) a^2 - \tfrac{1}{2} \lVert v \rVert^2$, $2k=\text{rk}(A+A^t)$ and
\[
S=-\left( \tfrac{k}{4} + \tfrac{1}{2} \right)a^2 \text{Id} - \tfrac{1}{2} AA^t + \tfrac{a}{4} \left(A + A^t\right).
\]
Notice that the  previous expression differs from the one in \cite{AL} by a sign inside the parenthesis in the first summand, which followed from a wrong formula in \cite[Proposition 3.1]{Vez} (\cite{VezPr}).
In particular,  for $v_0=0$, one has $v(t)=0$ for all $t$, and the system for the pair $(a,A)$ reduces to
\[
\begin{cases}
\dot{a}=-\left( \tfrac{k}{4} + \tfrac{1}{2} \right) a^3, & a(0)=a_0, \\
\dot{A}= -\left( \tfrac{k}{4} + \tfrac{1}{2} \right) a^2 A, & A(0)=A_0,
\end{cases}
\]
which has explicit solution
\[
(a(t),A(t))=(a_0,A_0) \cdot c(t),\quad c(t)=\frac{1}{ \sqrt{1+a_0^2\left( \tfrac{k}{2}+1\right)t}}.
\]

We then deduce that the examples of split generalized K\"ahler structures of Theorem \ref{th_SGK} are all expanding solitons with scaling factor $c(t)$.  By \cite[Theorem 4.14]{AL}, any other split generalized K\"ahler structure on these groups converges, in the Cheeger-Gromov sense and after a suitable normalization, to an expanding soliton.

\section{Appendix:  Six-dimensional almost abelian Lie algebras}
Here we provide the classification of six-dimensional non-nilpotent almost abelian Lie algebras. Table \ref{table-indecomp} features the indecomposable ones, whose classification was obtained in \cite{Mu3} and refined in \cite{Sha}. In Table \ref{table-decomp} one can find six-dimensional non-nilpotent almost abelian Lie algebras which can be decomposed as a direct sum of two or more Lie algebras: these were singled out by studying \cite{Mu,Mu2}.
For each Lie algebra  in Tables \ref{table-indecomp} and \ref{table-decomp}    we include the  conditions on the parameters (if any)  for which the algebra   is unimodular.

In Table  \ref{table-cpx}  we give an explicit  complex structure for every Lie algebra in Theorem \ref{CPX} (the  conditions on the parameters  involved in  the structure equations are given in Theorem \ref{CPX}). \newpage

\begin{table}[H]
\begin{center}
\addtolength{\leftskip} {-2cm}
    \addtolength{\rightskip}{-2cm}
\scalebox{0.75}{
\begin{tabular}{|l|l|l|l|l|}
\hline \xrowht{15pt}
Name& Structure equations & Conditions & Unimodular & Complex structure \\
\hline \hline \xrowht{20pt}
$\mathfrak{g}_{6.1}^{p,q,r,s}$ & $(f^{16},pf^{26},qf^{36},rf^{46},sf^{56},0)$ &  $1 \geq |p| \geq |q| \geq |r| \geq |s| > 0$ & $s=-1-p-q-r$ & \begin{tabular}{@{}l@{}} ($p=q$, $r=s$) or ($p=1$, $r=s$) \\ or ($p=1$, $q=r$) \end{tabular} \\
\hline \xrowht{20pt}
$\mathfrak{g}_{6.2}^{p,q,r}$ & $(f^{16},pf^{26}+f^{36},pf^{36},qf^{46},rf^{56},0)$ & $1 \geq |q| \geq |r| > 0$ & $r=-1-2p-q$ & $(p=1,\,q=r)$ or $(q=1,\,p=r)$ \\
\hline \xrowht{20pt}
$\mathfrak{g}_{6.3}^{p,q}$ & $(f^{16},pf^{26}+f^{36},pf^{36}+f^{46},pf^{46},qf^{56},0)$ & $1 \geq |q| > 0$ & $q=-1-3p$ & $-$ \\
\hline \xrowht{20pt}
$\mathfrak{g}_{6.4}^{p}$ & $(f^{16},pf^{26}+f^{36},pf^{36}+f^{46},pf^{46}+f^{56},pf^{56},0)$ &  & $p=-\frac{1}{4}$ & $-$ \\
\hline \xrowht{20pt}
$\mathfrak{g}_{6.5}$ & $(f^{16}+f^{26},f^{26}+f^{36},f^{36}+f^{46},f^{46}+f^{56},f^{56},0)$ &  & $-$ & $-$ \\
\hline \xrowht{20pt}
$\mathfrak{g}_{6.6}^{p,q}$ & $(f^{16},pf^{26}+f^{36},pf^{36},qf^{46}+f^{56},qf^{56},0)$ & $|p| \geq |q|$ & $q=-\frac{1}{2}-p$ & $p=q$ \\
\hline \xrowht{20pt}
$\mathfrak{g}_{6.7}^{p,q}$ & $(pf^{16}+f^{26},pf^{26}+f^{36},pf^{36},qf^{46}+f^{56},qf^{56},0)$ & $p^2+q^2\neq 0$ & $q=-\frac{3}{2}p$ & $p=q$ \\
\hline \xrowht{20pt}
$\mathfrak{g}_{6.8}^{p,q,r,s}$ & $(pf^{16},qf^{26},rf^{36},sf^{46}+f^{56},-f^{46}+sf^{56},0)$ & $|p| \geq |q| \geq |r| >0$ & $s=-\frac{1}{2}(p+q+r)$ & $p=q$ or $q=r$ \\
\hline \xrowht{20pt}
$\mathfrak{g}_{6.9}^{p,q,r}$ & $(pf^{16},qf^{26}+f^{36},qf^{36},rf^{46}+f^{56},-f^{46}+rf^{56},0)$ & $p \neq 0$ & $r=-\frac{1}{2}p-q$ & $p=q$ \\
\hline \xrowht{20pt}
$\mathfrak{g}_{6.10}^{p,q}$ & $(pf^{16}+f^{26},pf^{26}+f^{36},pf^{36},qf^{46}+f^{56},-f^{46}+qf^{56},0)$ &  & $q=-\frac{3}{2}p$ &$-$ \\
\hline \xrowht{20pt}
$\mathfrak{g}_{6.11}^{p,q,r,s}$ & $(pf^{16},qf^{26}+f^{36},-f^{26}+qf^{36},rf^{46}+sf^{56},-sf^{46}+rf^{56},0)$ & \begin{tabular}{@{}l@{}} 
$ps \neq 0$ \\
$|q| > |r| \text{ or } (|q|=|r|,$ $|s| \leq 1)$ \\
\end{tabular} & $r=-\frac{1}{2}p-q$ & \cmark \\
\hline \xrowht{20pt}
$\mathfrak{g}_{6.12}^{p,q}$ & $(pf^{16},qf^{26}+f^{36}-f^{46},-f^{26}+qf^{36}-f^{56},qf^{46}+f^{56},-f^{46}+qf^{56},0)$ & $p \neq 0$ & $q=-\frac{1}{4}p$ & \cmark \\ \hline
\end{tabular}}
\caption{Six-dimensional indecomposable non-nilpotent almost abelian Lie algebras.} \label{table-indecomp}
\end{center}
\end{table}

\begin{table}[H]
\begin{center}
\addtolength{\leftskip} {-2cm}
\addtolength{\rightskip}{-2cm}
\scalebox{0.75}{
\begin{tabular}{|l|l|l|l|l|}
\hline \xrowht{15pt}
Name& Structure equations & Conditions & Unimodular & Complex structure \\
\hline \hline  \xrowht{20pt}
$\mathfrak{g}_2 \oplus 4\R$ & $(f^{16},0,0,0,0,0)$ &  & $-$ & \cmark \\ \hline
\xrowht{20pt}
$\mathfrak{g}_{3.2} \oplus 3\R$ & $(f^{16}+f^{26},f^{26},0,0,0,0)$ &  & $-$ & $-$\\
\hline \xrowht{20pt}
$\mathfrak{g}_{3.3} \oplus 3\R$ & $(f^{16},f^{26},0,0,0,0)$  &  &$-$ & \cmark\\
\hline \xrowht{20pt}
$\mathfrak{g}_{3.4}^p \oplus 3\R$ & $(f^{16},pf^{26},0,0,0,0)$ & $1 \geq |p| > 0$, $p\neq 1$ & $p=-1$ & $-$ \\
\hline \xrowht{20pt}
$\mathfrak{g}_{3.5}^p \oplus 3\R$ & $(pf^{16}+f^{26},-f^{16}+pf^{26},0,0,0,0)$ &  & $p=0$ & \cmark \\
\hline
\xrowht{20pt}
$\mathfrak{g}_{4.2}^{p} \oplus 2\R$ & $(pf^{16},f^{26}+f^{36},f^{36},0,0,0)$ & $p \neq 0$ & $p=-2$ & $p=1$ \\
\hline \xrowht{20pt}
$\mathfrak{g}_{4.3} \oplus 2\R$ & $(f^{16},f^{36},0,0,0,0)$ &  & $-$ & $-$\\
\hline \xrowht{20pt}
$\mathfrak{g}_{4.4} \oplus 2\R$ & $(f^{16}+f^{26},f^{26}+f^{36},f^{36},0,0,0)$ &  &$-$ & $-$\\
\hline \xrowht{20pt}
$\mathfrak{g}_{4.5}^{p,q} \oplus 2\R$ & $(f^{16},pf^{26},qf^{36},0,0,0)$ & $1 \geq |p| \geq |q| > 0$ & $q=-1-p$ & $p=q$ or $p=1$\\
\hline \xrowht{20pt}
$\mathfrak{g}_{4.6}^{p,q} \oplus 2\R$ & $(pf^{16},qf^{26}+f^{36},-f^{26}+qf^{36},0,0,0)$ & $p \neq 0$ & $q=-\frac{p}{2}$ & \cmark \\
\hline
\xrowht{20pt}
$\mathfrak{g}_{5.7}^{p,q,r} \oplus \R$ & $(f^{16},pf^{26},qf^{36},rf^{46},0,0)$ & $1\geq|p|\geq |q| \geq |r| > 0$ & $r=-1-p-q$ & $p=1$, $q=r$ \\
\hline \xrowht{20pt}
$\mathfrak{g}_{5.8}^{p} \oplus \R$ & $(f^{16},pf^{26},f^{46},0,0,0)$ & $1\geq|p| > 0$ & $p=-1$ & $p=1$ \\
\hline \xrowht{20pt}
$\mathfrak{g}_{5.9}^{p,q} \oplus \R$ & $(pf^{16},qf^{26},f^{36}+f^{46},f^{46},0,0)$ & $|p|\geq|q| > 0$ & $q=-2-p$ & $-$\\
\hline \xrowht{20pt}
$\mathfrak{g}_{5.10} \oplus \R$ & $(f^{16},f^{36},f^{46},0,0,0)$ &  & $-$ & $-$ \\
\hline \xrowht{20pt}
$\mathfrak{g}_{5.11}^{p} \oplus \R$ & $(pf^{16},f^{26}+f^{36},f^{36}+f^{46},f^{46},0,0)$  & $p \neq 0$ & $p=-3$ & $-$\\
\hline \xrowht{20pt}
$\mathfrak{g}_{5.12} \oplus \R$ & $(f^{16}+f^{26},f^{26}+f^{36},f^{36}+f^{46},f^{46},0,0)$ &  & $-$ & $-$ \\
\hline \xrowht{20pt}
$\mathfrak{g}_{5.13}^{p,q,r} \oplus \R$ & $(f^{16},pf^{26},qf^{36}+rf^{46},-rf^{36}+qf^{46},0,0)$ & $1 \geq |p| > 0,$ $r \neq 0$ & $q=-\frac{1}{2}(1+p)$ & $p=1$ \\
\hline \xrowht{20pt}
$\mathfrak{g}_{5.14}^{p} \oplus \R$ & $(pf^{16}+f^{26},-f^{16}+pf^{26},f^{46},0,0,0)$ &  & $p=0$ & \cmark \\
\hline \xrowht{20pt}
$\mathfrak{g}_{5.15}^{p} \oplus \R$ & $(f^{16}+f^{26},f^{26},pf^{36}+f^{46},pf^{46},0,0)$ & $|p| \leq 1$ & $p=-1$ & $p=1$  \\
\hline \xrowht{20pt}
$\mathfrak{g}_{5.16}^{p,q} \oplus \R$ & $(f^{16}+f^{26},f^{26},pf^{36}+qf^{46},-qf^{36}+pf^{46},0,0)$ & $q \neq 0$ & $p=-1$ & $-$ \\
\hline \xrowht{20pt}
$\mathfrak{g}_{5.17}^{p,q,r} \oplus \R$ & $(pf^{16}+f^{26},-f^{16}+pf^{26},qf^{36}+rf^{46},-rf^{36}+qf^{46},0,0)$ & \begin{tabular}{@{}l@{}} 
$r \neq 0$ \\
$|p| > |q| \text{ or } (|p|=|q|,$ $|r| \leq 1)$  \\
\end{tabular} & $q=-p$ & \cmark \\ 
\hline \xrowht{20pt}
$\mathfrak{g}_{5.18}^{p} \oplus \R$ & $(pf^{16}+f^{26}-f^{36},-f^{16}+pf^{26}-f^{46},pf^{36}+f^{46},-f^{36}+pf^{46},0,0)$ &  & $p=0$ & \cmark \\
\hline
\end{tabular}}
\caption{Six-dimensional decomposable non-nilpotent almost abelian Lie algebras.}
 \label{table-decomp}
\end{center}
\end{table}

\begin{table}
\begin{center}
\addtolength{\leftskip} {-2cm}
\addtolength{\rightskip}{-2cm}
\scalebox{0.75}{
\begin{tabular}{|l|l|l|}
\hline \xrowht{15pt}
Name& Structure equations &Complex structure \\
\hline \hline
\xrowht{20pt}
$\mathfrak{k}_1^{p,r} =\mathfrak{g}_{6.1}^{p,p,r,r}$ & $(f^{16},pf^{26},pf^{36},rf^{46},rf^{56},0)$  & $J f_1 = f_6,\, J f_2 = f_3,\, Jf_4 = f_5$ \\ \hline
\xrowht{20pt}
$\mathfrak{k}_2^{q,r} =\mathfrak{g}_{6.1}^{1,q,r,r}$ & $(f^{16},f^{26},qf^{36},rf^{46},rf^{56},0)$  & $J f_1 = f_2,\, J f_3 = f_6,\, Jf_4 = f_5$ \\ \hline
\xrowht{20pt}
$\mathfrak{k}_{3}^{q,s} =\mathfrak{g}_{6.1}^{1,q,q,s}$ & $(f^{16},f^{26},qf^{36},qf^{46},sf^{56},0)$  & $J f_1 = f_2,\, J f_3 = f_4,\, Jf_5 = f_6$ \\ \hline
\xrowht{20pt}
$\mathfrak{k}_{4}^q=\mathfrak{g}_{6.2}^{1,q,q}$ & $(f^{16},f^{26}+f^{36},f^{36},qf^{46},qf^{56},0)$ & $Jf_2=f_1,\,Jf_3=f_6,\,Jf_4=f_5$ \\
\hline
\xrowht{20pt}
$\mathfrak{k}_{5}^p=\mathfrak{g}_{6.2}^{p,1,p}$ & $(f^{16},pf^{26}+f^{36},pf^{36},f^{46},pf^{56},0)$ & $Jf_1=f_4,\,Jf_2=f_5,\,Jf_3=f_6$ \\
\hline
\xrowht{20pt}
$\mathfrak{k}_{6}^{p}=\mathfrak{g}_{6.6}^{p,p}$ & $(f^{16},pf^{26}+f^{36},pf^{36},pf^{46}+f^{56},pf^{56},0)$  &$J f_1 = f_6, \, J f_2 = f_4, \, Jf_3 = f_5$ \\
\hline
\xrowht{20pt}
$\mathfrak{k}_{7}^p=\mathfrak{g}_{6.7}^{p,p}$ & $(pf^{16}+f^{26},pf^{26}+f^{36},pf^{36},pf^{46}+f^{56},pf^{56},0)$ & $J f_1 = f_4, \, J f_2 = f_5, \, Jf_3 = f_6$ \\ \hline
\xrowht{20pt}
$\mathfrak{k}_{8}^{p,q,s} =\mathfrak{g}_{6.8}^{p,q,q,s}$ & $(pf^{16},qf^{26},qf^{36},sf^{46}+f^{56},-f^{46}+sf^{56},0)$  & $J f_1 = f_6, \, J f_2 = f_3, \, Jf_4 = f_5$ \\
\hline
\xrowht{20pt}
$\mathfrak{k}_{9}^{p,r,s}=\mathfrak{g}_{6.8}^{p,p,r,s}$ & $(pf^{16},pf^{26},rf^{36},sf^{46}+f^{56},-f^{46}+sf^{56},0)$ & $J f_1 = f_2, \, J f_3 = f_6, \, Jf_4 = f_5$ \\ 
\hline
\xrowht{20pt}
$\mathfrak{k}_{10}^{p,r}=\mathfrak{g}_{6.9}^{p,p,r}$ & $(pf^{16},pf^{26}+f^{36},pf^{36},rf^{46}+f^{56},-f^{46}+rf^{56},0)$ & $Jf_2=f_1,\,Jf_3=f_6,\,Jf_4=f_5$ \\
\hline
\xrowht{20pt}
$\mathfrak{k}_{11}^{p,q,r,s}=\mathfrak{g}_{6.11}^{p,q,r,s}$ & $(pf^{16},qf^{26}+f^{36},-f^{26}+qf^{36},rf^{46}+sf^{56},-sf^{46}+rf^{56},0)$ & $J f_1 = f_6, \, J f_2 = f_3, \, Jf_4 = f_5$ \\
\hline
\xrowht{20pt}
$\mathfrak{k}_{12}^{p,q}=\mathfrak{g}_{6.12}^{p,q}$ & $(pf^{16},qf^{26}+f^{36}-f^{46},-f^{26}+qf^{36}-f^{56},qf^{46}+f^{56},-f^{46}+qf^{56},0)$  & $J f_1 = f_6, \, J f_2 = f_3, \, Jf_4 = f_5$ \\
\hline
\xrowht{20pt}
$\mathfrak{k}_{13}=\mathfrak{g}_2 \oplus 4\R$ & $(f^{16},0,0,0,0,0)$  & $J f_1 = f_6, \, J f_2 = f_3, \, Jf_4 = f_5$  \\
\hline \xrowht{20pt}
$\mathfrak{k}_{14}=\mathfrak{g}_{3.3} \oplus 3\R$ & $(f^{16},f^{26},0,0,0,0)$ &  $J f_1 = f_2, \, J f_3 = f_4, \, Jf_5 = f_6$ \\
\hline \xrowht{20pt}
$\mathfrak{k}_{15}^p=\mathfrak{g}_{3.5}^p \oplus 3\R$ & $(pf^{16}+f^{26},-f^{16}+pf^{26},0,0,0,0)$  & $J f_1 = f_2, \, J f_3 = f_4, \, Jf_5 = f_6$ \\
\hline
\xrowht{20pt}
$\mathfrak{k}_{16}=\mathfrak{g}_{4.2}^1 \oplus 2\R$ & $(f^{16},f^{26}+f^{36},f^{36},0,0,0)$ & $Jf_2=f_1,\,Jf_3=f_6,\,Jf_4=f_5$ \\ \hline
 \xrowht{20pt}
$\mathfrak{k}_{17}^p=\mathfrak{g}_{4.5}^{p,p} \oplus 2\R$ & $(f^{16},pf^{26},pf^{36},0,0,0)$   & $J f_1 = f_6,\, J f_2 = f_3,\, Jf_4 = f_5$ \\
\hline 
 \xrowht{20pt}
$\mathfrak{k}_{18}^q=\mathfrak{g}_{4.5}^{1,q} \oplus 2\R$ & $(f^{16},f^{26},qf^{36},0,0,0)$   & $J f_1 = f_2,\, J f_3 = f_6,\, Jf_4 = f_5$ \\
\hline
\xrowht{20pt}
$\mathfrak{k}_{19}^{p,q}=\mathfrak{g}_{4.6}^{p,q} \oplus 2\R$ & $(pf^{16},qf^{26}+f^{36},-f^{26}+qf^{36},0,0,0)$  & $J f_1 = f_6, \, J f_2 = f_3, \, Jf_4 = f_5$ \\
\hline
\xrowht{20pt}
$\mathfrak{k}_{20}^{q}=\mathfrak{g}_{5.7}^{1,q,q} \oplus \R$ & $(f^{16},f^{26},qf^{36},qf^{46},0,0)$  & $J f_1 = f_2, \, J f_3 = f_4, \, Jf_5 = f_6$ \\
\hline \xrowht{20pt}
$\mathfrak{k}_{21}=\mathfrak{g}_{5.8}^1 \oplus \R$ & $(f^{16},f^{26},f^{46},0,0,0)$ &  $J f_1 = f_2, \, J f_3 = f_5, \, Jf_4 = f_6$ \\
\hline \xrowht{20pt}
$\mathfrak{k}_{22}^{q,r}=\mathfrak{g}_{5.13}^{1,q,r} \oplus \R$ & $(f^{16},f^{26},qf^{36}+rf^{46},-rf^{36}+qf^{46},0,0)$  & $J f_1 = f_2, \, J f_3 = f_4, \, Jf_5 = f_6$ \\
\hline \xrowht{20pt}
$\mathfrak{k}_{23}^p=\mathfrak{g}_{5.14}^p \oplus \R$ & $(pf^{16}+f^{26},-f^{16}+pf^{26},f^{46},0,0,0)$ &$J f_1 = f_2, \, J f_3 = f_5, \, Jf_4 = f_6$ \\ \hline
\xrowht{20pt}
$\mathfrak{k}_{24}=\mathfrak{g}_{5.15}^1 \oplus \R$ & $(f^{16}+f^{26},f^{26},f^{36}+f^{46},f^{46},0,0)$ & $J f_1 = f_3, \, J f_2 = f_4, \, Jf_5 = f_6$ \\
\hline \xrowht{20pt}
$\mathfrak{k}_{25}^{p,q,r}=\mathfrak{g}_{5.17}^{p,q,r} \oplus \R$ & $(pf^{16}+f^{26},-f^{16}+pf^{26},qf^{36}+rf^{46},-rf^{36}+qf^{46},0,0)$  & $J f_1 = f_2, \, J f_3 = f_4, \, Jf_5 = f_6$ \\
\hline \xrowht{20pt}
$\mathfrak{k}_{26}^p=\mathfrak{g}_{5.18}^p \oplus \R$ & $(pf^{16}+f^{26}-f^{36},-f^{16}+pf^{26}-f^{46},pf^{36}+f^{46},-f^{36}+pf^{46},0,0)$  & $J f_1 = f_2, \, J f_3 = f_4, \, Jf_5 = f_6$  \\ \hline
\end{tabular}}
\medskip \medskip
\caption{Six-dimensional non-nilpotent almost abelian Lie algebras admitting a complex structure.} \label{table-cpx}
\end{center}
\end{table}

\end{document}